\documentclass[12pt]{amsart}
\usepackage{amsmath, amssymb, amsthm, esint, verbatim, hyperref, enumerate, enumitem}
\usepackage[letterpaper,margin=1.1in]{geometry}

\hypersetup{
  pdfpagelayout=SinglePage,
  pdfpagemode=UseOutlines,
  colorlinks,
  bookmarksopen,
  linkcolor=[rgb]{0,0,0.7},
  urlcolor=[rgb]{0,0,0.4},
  citecolor=[rgb]{0.4,0.1,0},
}

\newcommand{\dist}{\mathrm{dist}}

\newcommand{\spt}{\mathrm{spt}}
\newcommand{\haus}{\mathcal{H}}
\newcommand{\leb}{\mathcal{L}}

\newcommand{\sing}{\mathrm{sing}}
\newcommand{\cT}{\mathcal{T}}
\newcommand{\cG}{\mathcal{G}}
\newcommand{\cS}{\mathcal{S}}
\newcommand{\cB}{\mathcal{B}}

\newcommand{\cF}{\mathcal{F}}
\newcommand{\cU}{\mathcal{U}}

\newcommand{\R}{\mathbb R}
\newcommand{\fb}{{ \{ u > 0 \} }}
\newcommand{\eps}{\epsilon}
\newcommand{\reg}{\mathrm{reg}}

\newtheorem{theorem}{Theorem}[section]
\newtheorem*{theorem*}{Theorem}

\newtheorem{proposition}[theorem]{Proposition}

\newtheorem{lemma}[theorem]{Lemma}
\newtheorem{corollary}[theorem]{Corollary}
\theoremstyle{definition}
\newtheorem{definition}[theorem]{Definition}
\theoremstyle{remark}
\newtheorem{remark}{Remark}[section]
\theoremstyle{remark}

\theoremstyle{remark}

\theoremstyle{remark}

\theoremstyle{remark}

\usepackage{xcolor}

\title{Quantitative stratification for some free-boundary problems}
\author{Nick Edelen and Max Engelstein}

\thanks{N.~Edelen was supported by NSF grant DMS-1606492. M.~Engelstein was partially supported by NSF Grant No. DMS-1440140 while the author was in residence at MSRI in Berkeley, California, during Spring 2017. The authors  thank Jeff Cheeger and Robin Neumayer for pointing out the references \cite{che-na} and \cite{kriv-lin} respectively.  We also thank Tatiana Toro for helpful comments on an earlier draft of this paper. Finally, we thank an anonymous referee for their comments.}
\address{Department of Mathematics\\Massachusetts Institute of Technology\\Cambridge, MA, 02139-4307 }
\email{nedelen@mit.edu and maxe@mit.edu}

\begin{document}

\begin{abstract}
In this paper we prove the rectifiability of and measure bounds on the singular set of the free boundary for minimizers of a functional first considered by Alt-Caffarelli \cite{alt-caf}. Our main tools are the Quantitative Stratification and Rectifiable-Reifenberg framework of Naber-Valtorta \cite{naber-valtorta:harmonic}, which allow us to do a type of ``effective dimension-reduction." The arguments are sufficiently robust that they apply to a broad class of related free boundary problems as well. 
\end{abstract}

\maketitle
\tableofcontents

\section{Introduction}

Let us consider a function $u : \Omega \subset \R^n \to \R$ minimizing the functional
\begin{equation}\label{altcaffunctional}\tag{$\star_{\Omega, Q}$}
J(u) = \int_\Omega |Du(x)|^2 + Q^2(x)1_{\{u > 0\}}(x)\ dx, 
\end{equation}
subject to the condition $u|_{\partial\Omega} = u_0|_{\partial\Omega} \geq 0$ for some $u_0 \in W^{1,2}(\Omega)$, and \emph{positive} $Q \in C^{\alpha}(\Omega)$.  We will abuse notation and also refer to this minimization problem as $(\star_{\Omega, Q})$.

For unbounded $\Omega$, let us say $u$ solves $(\star_{\Omega, Q})$ if $u|_{\partial\Omega} = u_0|_{\partial\Omega}$ and $u$ minimizes $J(u)$ on every compact subset, in the sense that for every $D \subset \subset \Omega$, we have
\[
\int_D |Du|^2 + Q^2 1_{\fb} \leq \int_D |Dv|^2 + Q^2 1_{\fb} \quad \forall v \text{ s.t. } u - v \in W^{1,2}_0(D).
\]

For $u$ as above, $\partial \{ u > 0 \}$ is called the \emph{free-boundary}, and $(\star_{\Omega, Q})$ is a simple example (``one-phase with constant coefficients'') in a broad class of free-boundary problems.  These problems arise naturally in physical models, such as heat flows and jet cavitation.  Other fundamental types of free-boundary problems, which we shall also address in this paper, are the two-phase, vector-valued one-phase, and one-phase almost-minimizers (see Section \ref{section:other-probs}).

The problem $(\star_{\Omega, Q})$ was first studied systematically by Alt-Caffarelli in their seminal paper \cite{alt-caf}.  They demonstrated existence and certain regularity of $u$ and its free-boundary, and proved that $u$ essentially solves the system:
\[
u \geq 0, \quad \Delta u = 0 \text{ on } \{ u > 0 \} \cap \Omega, \quad |\nabla u| = Q \text{ on } \partial \{u > 0 \} \cap \Omega . 
\]
Alt-Caffarelli showed the free-boundary is a set of locally finite perimeter, and proved an $\eps$-regularity (also known as ``improvement of flatness") theorem for the free-boundary (see Theorem \ref{thm:altcaf}).

The \emph{regular set} of $\partial \{ u > 0 \}$, written $\reg(u)$, is defined to be the set of points in $\partial \{u > 0 \}$ near which $\partial \{u > 0\}$ can be written as a $C^{1,\beta}$ graph, for some $\beta > 0$.  The \emph{singular set}, which we write as $\sing(u)$, consists  of all other points in $\partial \{ u > 0\}$.  By the regularity theorem of \cite{alt-caf}, the regular set is open, dense, and has full $\haus^{n-1}$-measure.  Moreover, if $Q$ is smooth then the regular set is smooth also.

In general the singular set may be non-empty, but following a strong analogy between solutions of $(\star_{\Omega, Q})$ and minimal surfaces, significant progress has been made towards \emph{partial} regularity.  That is, demonstrating some kind of smallness of the singular set.  Let us summarize the known results in this direction: given $u$ solving $(\star_{\Omega, Q})$, then
\begin{enumerate}[label=(\roman*)]
\item For $n = 2$, $\sing(u) = \emptyset$ (Alt-Caffarelli \cite{alt-caf}).

\item Let $k^*$ be the first dimension admitting a non-linear, one-homogeneous solution of $(\star_{\R^{k^*}, 1})$.  Then we have the Hausdorff dimension bound $\dim \sing(u) \leq n-k^*$ (Weiss, \cite{weiss:partial-reg}).

\item $k^* \geq 4$ (Caffarelli-Jerison-Kenig \cite{cafjerken}).

\item $k^* \geq 5$ (Jerison-Savin \cite{JerisonSavinStability}). 

\item $k^* \leq 7$ (De Silva-Jerison \cite{desilvajersingular}).
\end{enumerate}
The exact value of $k^* \in \{5, 6, 7\}$, and hence the sharp dimension bound on the singular set, is still unknown.  Following the analogy with area-minimizing hypersurfaces, one might expect $k^* = 7$ (supported by the fact that the De Silva-Jerison example is closely related to the Simons cone).  We mention also that the central contribution of Weiss \cite{weiss:partial-reg} was to introduce a monotone quantity (see Section \ref{section:background}), which enforces one-homogeneity of blow-ups, and allowed him to adapt standard dimension reducing techniques of Federer-Almgren to the free-boundary setting.

In this paper we are interested in the \emph{quantitative} and \emph{fine-scale} structure of the singular set.  We shall adapt the techniques of Naber-Valtorta \cite{naber-valtorta:harmonic} to prove rectifiability and packing bounds on $\sing(u)$. For example, we obtain the following Theorem.  We recall that a set is $k$-rectifiable if $\haus^k$-almost-all of it is contained in the countable union of $C^1$ $k$-manifolds.
\begin{theorem}\label{thm:teaser}
Let $u$ solve $(\star_{\Omega, Q})$.  Then for every $D \subset\subset D'\subset\subset \Omega$, we have
\[
\haus^{n-k^*}(\sing(u) \cap D) \leq C,
\]
and $\sing(u) \cap D$ is $(n-k^*)$-rectifiable.  Here $C$ depends only on the quantities
\[
n, \, |Q|_{C^\alpha(D')}, \,  \min_{D'} Q , \, \dist(D, \partial D'), \, \leb^n(D').
\]
\end{theorem}

More generally, we shall extend the notions of \emph{quantitative stratification} to solutions of $(\star_{\Omega, Q})$.  The $k$-stratum is defined classically in terms of degrees of symmetry of the blow-ups (see Section \ref{section:strata}), and, through $\eps$-regularity Theorems, is directly related to the singular set.  As introduced in \cite{che-na}, we define for $(\star_{\Omega, Q})$ the $(k, \eps, r)$-strata, which are effective notions of the $k$-strata.  Theorem \ref{thm:teaser} is a Corollary of more precise packing estimates and rectifiability of the effective stata.

We also observe that many of our results here actually hold for a whole class of related free boundary problems. In particular, they hold (with only minor modifications) for the two-phase problem first considered by Alt-Caffarelli-Friedman \cite{alt-caf-fried}, the vectorial version of the Alt-Caffarelli problem (which has been an object of great interest recently, see \cite{caf-sha-yer} and \cite{maz-ter-vel}), and for almost-minimizers (in the sense of David and Toro, see \cite{davidtoroalmostminimizers}).  For simplicity we shall write this paper only for the one-phase problem with H\"older continuous $Q$, but the proofs carry over esssentially verbatim to other problems.

The techniques of Naber-Valtorta are very powerful, and should perhaps be seen as a kind of quantitative dimension reducing.  They have been used to prove similar results for harmonic maps \cite{naber-valtorta:harmonic}, varifolds with bounded mean curvature \cite{naber-valtorta:varifold}, approximate harmonic maps \cite{naber-valtorta:approx-harmonic}, and Q-valued harmonic maps \cite{delellis:q-valued}.  We feel the techniques are sufficiently new and involved that it would benefit readers to see their adaptation to the free boundary setting in significant detail.

\subsection{Notation and background} \label{section:background}

In this section we set up our notation, and recollect various background results.

We always work in $\R^n$.  Given a set $A \subset \R^n$, let $d(x, A)$ be the Euclidean distance from $x$ to $A$, and write
\[
B_r(A) = \{ x \in \R^n : d(x, A) < r \}
\]
for the open $r$-tubular neighborhood about $A$.  We shall write $|A| = \leb^n(A)$ for the $n$-volume of $A$, and $\haus^k(A)$ for the $k$-dimensional Hausdorff measure.  We let $d_H$ denote the usual Hausdorff distance between sets, so
\[
d_H(A_1, A_2) = \inf \{ r : A_1 \subset B_r(A_2) \text{ and } A_2 \subset B_r(A_1) \}.
\]
We write $\omega_n = |B_1(0)|$.

Given points $y_0, \ldots, y_k \in \R^n$, we let $<y_0, \ldots, y_k>$ be the affine $k$-space spanning $y_0, \ldots, y_k$.

Given a Borel measure $\mu$, the \emph{Jones' $\beta^k_2$-number} measures a one-sided $L^2$ $\mu$-distance to $k$-planes:
\begin{equation}\label{eqn:beta}
\beta^k_{\mu,2}(x, r)^2 = \inf_{V^k} r^{-k-2} \int_{B_r(x)} d(z, L)^2 d\mu(z) ,
\end{equation}
where the infimum is over \emph{affine} $k$-planes $V^k$.  We write $V^k_{\mu, 2}(x, r)$ for a choice of $k$-plane realizing $\beta$.  When unambiguous we may simply write $\beta(x, r)$.

Take a domain $\Omega \subset \R^n$.  We say $\Omega$ is a Lipschitz domain if the boundary $\partial \Omega$ can be written locally as a Lipschitz graph.  Given $f : \Omega \to \R$, we write $|f|_{C^0(\Omega)}$, $|f|_{C^\alpha(\Omega)}$, $||f||_{L^p(\Omega)}$, $||f||_{W^{1,p}(\Omega)}$ for the usual sup-, Holder-, $L^p$-, and Sobolev-norms.  We shall use the Holder semi-norm
\begin{align*}
[f]_{\alpha,\Omega} = \sup_{x \neq y \in \Omega} \frac{|f(x) - f(y)|}{|x - y|^\alpha} .
\end{align*}

Similarly, $C^0(\Omega)$, $C^\alpha(\Omega)$, $L^p(\Omega)$, $W^{1,2}(\Omega)$ are the usual function spaces.  $C^\infty_0(\Omega)$ is the space of compactly supported $C^\infty$ functions in $\Omega$, and $W^{1,p}_0(\Omega)$ is the closure of $C^\infty_0(\Omega)$ with respect to $||\cdot ||_{W^{1,p}(\Omega)}$.

We can and shall assume any constant $c$ or $c_i$ is at least $1$.  The precise value of a constant written $c$ (without any subscript) may increase from line to line.

\vspace{2.5mm}

Alt-Caffarelli have shown certain (quantitative) regularity of $u$, depending only on $Q$ and $d(\cdot, \partial\Omega)$.  We summarize the relevant results below.  Since $Q$ is positive, for shorthand we will often write $|Q + 1/Q|_{C^0} \leq \Lambda$ in place of $1/\Lambda \leq \min Q \leq \max Q \leq \Lambda$.
\begin{theorem}[Basic existence, regularity \cite{alt-caf}]\label{thm:altcaf}
Let $u$ solve $(\star_{\Omega, Q})$, and suppose $|Q + 1/Q|_{C^0(\Omega)} \leq \Lambda$.  Let us fix an $x \in \fb \cap \Omega$ with $2r < d(x, \partial\Omega)$.  Then the following hold:
\begin{enumerate}
\item (existence) If $\Omega$ is a Lipschitz domain, and $J(u_0) < \infty$, then the minimizer $u$ exists.

\item (representation) The free-boundary $\partial \fb$ has locally-finite $\haus^{n-1}$-measure, and $\Delta u = q \haus^{n-1} \llcorner \partial \fb$ in the sense of distributions, for some measurable function $q$ such that $q = Q$ $\haus^{n-1}$-a.e. in $\partial \fb$.

\item (Lipschitz control) We have
\[
||Du||_{L^\infty(B_r(x))} \leq c(n)\Lambda.
\]

\item (non-degeneracy) For $y \in B_r(x)$ we have
\[
\frac{1}{c(n,\Lambda)} d(y, \partial \fb) \leq u(y) \leq c(n,\Lambda) d(y, \partial \fb) .
\]

\item ($\eps$-regularity) Assume $Q\in C^\alpha(\Omega)$. There are constants $\eps_0(n)$, and $\beta(n, \alpha, \Lambda) < 1$ so that if
\begin{equation}\label{eqn:eps-reg-cond}
\partial \fb \cap B_r(x) \subset B_{\eps r}(V^{n-1})
\end{equation}
for some affine $(n-1)$-plane $V^{n-1}$, then $\partial \fb \cap B_{r/2}(x)$ is a $C^{1,\beta}$ graph over $V$.
\end{enumerate}
\end{theorem}

As suggested by Theorem \ref{thm:altcaf}(4), the natural scaling for $(\star)$ is graph dilation.  In other words, if we define
\[
u_{y, r}(x) = r^{-1} u(y + rx),
\]
then $u$ solves $(\star_{B_r(y), Q})$ if and only if $u_{y, r}$ solves $(\star_{B_1(0), r Q_{y, r}})$.  Notice that $Q$ scales like $Du$, and in particular as $r \to 0$ we have $[r Q_{y, r}]_{\alpha, B_1(0)} \to 0$.  So at small scales $u$ looks like a solution to $(\star)$ with constant $Q$.


An essentially direct consequence of Theorem \ref{thm:altcaf}, Alt-Caffarelli proved the following compactness Theorem (see also Theorem 9.1 in \cite{davidtoroalmostminimizers}). 

\begin{theorem}[Compactness \cite{alt-caf}]\label{thm:compactness}
Let $u_i$ be a sequence of solutions to $(\star_{B_{r_i}(0), Q_i})$, with $0 \in \partial\{ u_i > 0 \}$ for every $i$.  Suppose $r_i \to r \in (0, \infty]$, and  
\[
0 < \inf_i \min_{D} Q_i \leq \sup_i \max_{D} Q_i < \infty \quad \forall D \subset\subset B_r(0),
\]
and $Q_i \to Q$ in $C^0_{loc}(B_r(0))$.

Then there is a subsequence $i'$, and Lipschitz function $u_\infty \in W^{1,\infty}_{loc}(B_r(0))$, so that $u_{i'} \to u$ in the following senses:
\begin{enumerate}
\item $u_{i'} \to u_\infty$ in $C^{\alpha}_{loc}$ for every $0 \leq \alpha < 1$, and in $L^2_{loc}$, 

\item $D u_{i'} \to Du_\infty$ weak-$\star$-ly in $L^\infty_{loc}$, and strongly in $L^2_{loc}$, 

\item $\partial \{ u_{i'} > 0 \} \to \partial \{ u_\infty > 0 \}$ in the local Hausdorff distance,

\item $1_{\{u_{i'} > 0 \}} \to 1_{\{ u_\infty > 0 \}}$ in $L^1_{loc}$.
\end{enumerate}

Moreoever, $u_\infty$ solves $(\star_{B_r(0), Q})$ on compact sets, in the sense that for every $D \subset\subset B_r(0)$, we have
\[
\int_D |Du_\infty|^2 + Q^2 1_{\{ u_\infty > 0 \}} \leq \int_D |Dv|^2 + Q^2 1_{\{v > 0\}} \quad \forall v \text{ s.t. } u - v \in W^{1,2}_0(D).
\]
\end{theorem}

\begin{remark}
All the properties except strong $W^{1,2}_{loc}$ convergence are stated and proven in \cite{alt-caf}.  The strong $L^2_{loc}$ convergence of $Du_i$ is essentially a direct consequence of smooth convergence away from the free-boundaries, and boundedness of the $Du_i$ (see Theorem 9.1 in \cite{davidtoroalmostminimizers} for the full details in the almost-minimizers setting).
\end{remark}

\begin{remark}\label{rem:sing-no-die}
A simple but important Corollary of $\eps$-regularity is that singular points do not disappear under limits: if $u_i \to u$ as in Theorem \ref{thm:compactness}, and if $x_i \in \sing(u_i)$, $x_i \to x \in B_r(0)$, then $x \in \sing(u)$ also.

\end{remark}

\vspace{2.5mm}

In \cite{weiss:partial-reg} Weiss discovered a monotone quantity, which enforces a $1$-homogeneity in blow-ups.  The \emph{Weiss density} of $u$ in $B_r(y)$ is defined to be
\begin{equation}\label{eqn:weissmonodef}
W_r(Q, u, y) = r^{-n} \int_{B_r(y)} |Du|^2 dx + Q^2(y) r^{-n} |\fb\cap B_r(y)| - r^{-n-1} \int_{\partial B_r(y)} u^2 d\sigma.
\end{equation}
Notice in \eqref{eqn:weissmonodef} we are \emph{fixing} the value of $Q$ at $y$.  When unambiguous we may write $W_r(x)$ for $W_r(Q, u, x)$.

The density $W$ is scale-invariant, in the following sense:
\begin{equation}\label{eqn:wrescaled}
W_r(Q, u, y) = W_1(rQ_{y, r}, u_{y, r}, 0).
\end{equation}
Let us also remark that from Theorem \ref{thm:altcaf}, $W$ is controlled by $|Q|_{C^0}$ and $d(\cdot, \partial\Omega)$.  Precisely, whenever $x \in \partial \fb$ and $B_{5r}(x) \subset\Omega$, then
\begin{equation}\label{eqn:wcontrolled}
\sup_{y \in B_{2r}(x)} W_{2r}(Q, u, y) \leq c(n)(1 + |Q|_{C^0(B_{5r}(y))}).
\end{equation}

Weiss demonstrated the following monotonicity.
\begin{theorem}[Monotonicity \cite{weiss:partial-reg}]\label{thm:weiss-monotonicity}
Let $u$ solve $(\star_{\Omega, Q})$.  Given any $x \in \partial \fb \cap \Omega$, then for any $0 < s < r < R < d(x, \partial\Omega)$, we have
\begin{equation}\label{eqn:wdefect}
\int_{B_r(x) \setminus B_s(x)} \frac{ |u - (y-x)\cdot Du|^2}{|y - x|^{n+2}} dy \leq W_r(u, x) - W_s(u, x) + c_0(n,\alpha) r^\alpha [Q]_{\alpha, B_R(x)}.
\end{equation}

In particular, when $Q$ is constant then $W$ is increasing in $r$, and strictly increasing unless $u$ is $1$-homogenous.
\end{theorem}

From (almost-)monotonicity we have a well-defined notion of density of a point:
\[
W_0(Q, u, x) := \lim_{r\to 0} W_r(Q, u, x).
\]
Moreover, $W_0$ is upper-semi-continuous in the following sense: if $u_i$, $u$ are solutions to $(\star_{\Omega_i, Q_i})$ with uniform $C^\alpha$ bounds on $Q_i$, and $u_i \to u$ as in Theorem \ref{thm:compactness}, and $x_i \to x$, and $r_i \to 0$, then
\[
W_0(Q, u, x) \geq \limsup_i W_{r_i}(Q_i, u_i, x_i) \geq \limsup_i W_0(Q_i, u_i, x_i).
\]
Of course, for any fixed positive radius $r > 0$, with $B_r(x) \subset\subset \Omega$, $W_r$ is continuous with respect to this convergence:
\[
W_r(Q, u, x) = \lim_{i \to \infty} W_r(Q_i, u_i, x_i).
\]

Monotonicity and compactness gives the existence of $1$-homogenous ``blow-ups'' (or ``tangent solutions") at any point in the free-boundary.  If $x_0 \in \partial \fb \cap \Omega$, and $r_i \to 0$ is some sequence, then (after passing to a subsequence) we have convergence $u_{x_0, r_i} \to u_\infty$ as in Theorem \ref{thm:compactness}.  We call a $u_\infty$ obtained in this fashion the \emph{blow-up} at $x_0$.

The resulting $u_\infty$ will be non-zero, one-homogenous, and solve $(\star_{\R^n, Q(x_0)})$, in the sense that for every ball $B_\rho(y)$, 
\[
\int_{B_\rho(y)} |Du_\infty|^2 + Q(x_0)^2 1_{\{u_\infty > 0\}} \leq \int_{B_\rho(y)} |Dv|^2 + Q(x_0)^2 1_{\{ v> 0\}}  \quad \forall v \text{ s.t. } u - v \in W^{1,2}_0(B_\rho(y)).
\]
Homogeneity follows from Weiss's monotonicity, since $Q_\infty \equiv Q(x_0)$ is constant, and
\[
W_R(Q(x_0), u_\infty, 0) = \lim_{r_i \to 0} W_R(r_i Q_{x, r_i}, u_{x, r_i}, 0) = W_0(Q, u, x) \quad \forall R > 0.
\]
Notice that $\partial\{u_\infty > 0\}$ will be an $(n-1)$-cone.  We also remark that it is an open problem whether blow-ups are unique in general; of course at regular points $u_\infty$ is unique, and up to rotation is equal to 
\[
u_\infty(x) = Q(x_0)^2 \max\{ x_n, 0 \}.
\]

An easy calculation shows that for blow-ups $u_\infty$ as above, the Weiss density of $u_\infty$ at $0$, and hence the density of $u$ at $x_0$, is equal to 
\[
W_0(Q, u, x_0) \equiv W_1(Q(x_0), u_\infty, 0) = Q(x_0)^2 | \{ u_\infty > 0\} \cap B_1(0) | .
\]
It was observed by Beckner-Kenig-Pipher (in unpublished work, see also \cite{cafken}) that Faber-Krahn in the sphere and $\eps$-regularity imply that
\[
W_0(Q, u, x_0) \geq Q(x_0)^2 \frac{\omega_n}{2},
\]
with equality if and only if $x_0$ is regular.  A straightforward contradiction argument, using Theorem \ref{thm:compactness} on $1$-homogenous solutions, can rephrase $\eps$-regularity of Alt-Caffarelli in terms of density.

\begin{proposition}[\cite{alt-caf} with \cite{weiss:partial-reg}]\label{prop:reg-by-W}
There is an $\eps_0(n) > 0$ so that the following holds: if $u$ solves $(\star_{\Omega, Q})$, then $x \in \partial \fb \cap \Omega$ is regular if and only if $W_0(Q, u, x) < Q(x)^2 \frac{\omega_n}{2} (1 + \eps_0)$.
\end{proposition}

\subsection{Quantative stratification}\label{section:strata}

We recall the notion of the $k$-dimensional strata, and define the effective $k$-strata.  The technique of quantitative stratification was first introduced by Cheeger-Naber \cite{che-na} to study manifolds with Ricci curvature bounded from below, and has subsequently been used in a variety of geometric and analytic contexts.

Each $k$-stratum captures where $u$ looks at most $k$-dimensional on an infinitesimal scale, or equivalently, where $u$ doesn't look $(k+1)$-dimensional at infinitesimal scales.  The effective $k$-strata consider regions where $u$ looks quantitatively far from $(k+1)$-dimensional, on a local scale.  Enforcing structure on $u$ at a fixed positive scale is crucial to obtaining effective estimates.

\begin{definition}
A function $u$ is \emph{$k$-symmetric} if $u$ is $1$-homogeneous about some point (i.e. $u(x + r\xi) = ru(x + \xi)$ for all $\xi \in \mathbb S^{n-1}$), and there is a $k$-dimensional plane, $L^k$, so that
\[
u(x + v) = u(x), \quad \forall v \in L^k.
\]
\end{definition}

\begin{definition}
Given $x$, we say $u$ is \emph{$(k, \eps)$-symmetric} in $B_r(x)$ if
\[
r^{2-n} \int_{B_r(x)} |u - \tilde u|^2\ dy < \eps,
\]
for some $k$-symmetric $\tilde u$.
\end{definition}

\begin{definition}\label{stratadef}
Let $u$ solve $(\star_{\Omega, Q})$.  The \emph{$k$-stratum}, $S^k(u)$, is the set of points $x \in \partial\{u > 0\} \cap \Omega$ for which every blow-up at $x$ is at most $k$-symmetric.

\vspace{1mm}
We define \emph{$(k, \eps)$-stratum}, $S^k_\eps(u)$, to be the set of points $x \in \partial\{u > 0\} \cap \Omega$ for which $u$ is \underline{not} $(k+1,\eps)$-symmetric in $B_s(x)$, for every $0 < s \leq \min\{1, d(x, \partial\Omega)\}$.

\vspace{1mm}
We will further quantify the strata by defining the \emph{$(k, \eps, r)$-stratum} $S^k_{\eps, r}(u)$ to be the set of points $x \in \partial \fb \cap \Omega$ for which $u$ is not $(k+1, \eps)$-symmetric in $B_s(x)$, for every $r \leq s \leq \min\{1, d(x, \partial\Omega)\}$.
\end{definition}

It is instructive to make some simple observations concerning the above definition.
\begin{enumerate}[label=(\roman*)]
\item The $k$-strata satisfy
\[
S^0 \subset \ldots \subset S^{n-2} \subset S^{n-1} = S^n = \partial \fb \cap \Omega,
\]
where the penultimate equality is due to non-degeneracy of solutions.  Clearly, we also have
\[
reg(u) \subset S^{n-1} \setminus S^{n-2}.
\]

\item\label{itm:inclusions} If $\delta < \eps$ and $r < s$, then $S^k_{\eps, r} \subset S^k_{\delta, s}$.  If $j < k$, then $S^j_{\eps, r} \subset S^k_{\eps, r}$.

\item\label{itm:unions} We have $S^k_\eps = \cap_{r > 0} S^k_{\eps, r}$, and $S^k= \cup_{\eps > 0} S^k_{\eps}$ by an easy contradiction argument.

\item\label{itm:closure} Unlike the $S^k$, the $S^k_\eps$ are closed, in both $x$ and $u$.  In other words, if $u_i \to u$ in $L^2$, and $x_i \to x$, with $x_i \in S^k_{\eps}(u_i)$, then $x \in S^k_{\eps}(u)$.

\item\label{itm:limits} If $u_i \to u$ in $L^2$, and each $u_i$ is $(k, \eps_i)$-symmetric in $B_1(0)$, with $\eps_i \to 0$, then $u$ is $k$-symmetric in $B_1(0)$.
\end{enumerate}

The main utility of the $k$-strata has been the dimension bound due to Almgren-Federer, adapted to the free-boundary setting by Weiss:
\[
\dim(S^k(u)) \leq k.
\]
The fundamental observation behind dimension reducing is the following: if $u$ is $1$-homogenous, and $k$-symmetric along the $k$-plane $L^k$, then any blow-up away from $L^k$ will be $(k+1)$-symmetric.

Demonstrating an inclusion $\sing(u) \subset S^k(u)$ gives directly a dimension bound on the singular set.  Recalling that $k^*$ was the first dimension admitting non-trivial, $1$-homogenous solutions to $(\star_{\R^{k^*}, 1})$, we have by $\eps$-regularity that $\sing(u) \subset S^{n-k^*}(u)$.

In fact, the $\eps$-regularity of Alt-Caffarelli allows us to demonstrate $\sing(u) \subset S^{n-k^*}_\eps(u)$, for some positive $\eps$.  This is essentially due to the fact whenever the free-boundary looks flat \emph{at any scale}, then it is regular.
\begin{proposition}\label{prop:singsetinsidestrata}
There is an $\eps(n, \Lambda, \alpha)$ so that if $u$ solves $(\star_{B_3(0), Q})$ with $|Q + 1/Q|_{C^0(B_3(0))} \leq \Lambda$ and $[Q]_{\alpha, B_3(0)} \leq \Lambda$, then $\sing(u) \cap B_1(0) \subset S^{n-k^*}_\eps(u)$.
\end{proposition}

\begin{proof}
This follows by $\eps$-regularity using a straightforward compactness argument.  Suppose, towards a contradiction, we have a sequence $u_i$ solving $(\star_{B_3(0), Q_i})$, and collections $\eps_i \to 0$, $x_i \in B_1(0)$, $r_i \in (0, 1]$ satisfying the hypotheses, but so that $u_i$ is $(n-k^* + 1, \eps_i)$-symmetric in $B_{r_i}(x_i)$.

By translation and dilation invariance, we can assume $x_i \equiv 0$, $r_i \equiv 1$, and the $u_i$ solve $(\star_{B_2(0), Q_i})$, with $|Q_i + 1/Q_i|_{C^0(B_2(0))} \leq \Lambda$ and $[Q_i]_{\alpha, B_2(0)}  \leq \Lambda$.  This last inequality follows because we are making a positive dilation (so, $[Q]_\alpha$ will only decrease).

Passing to a subsequence, we obtain convergence $Q_i \to Q$ in $C^0(B_1(0))$ and $u_i \to u$ as in Theorem \ref{thm:compactness}.  The resulting $u$ will be $(n-k^*+1, 0)$-symmetric in $B_1(0)$, and solve $(\star_{B_{1}(0), Q})$, and have $0 \in \partial \fb$.

In particular, since $u$ is $1$-homogenous at some point, any blow-up $u'$ at $0$ will be $(n-k^*+1, 0)$-symmetric also, and solve $(\star_{\R^n, Q(0)})$.  Therefore $u'$ must be linear, as otherwise we could obtain a non-linear, $1$-homogenous solution in $\R^{k^*-1}$.  We deduce $0$ is a regular point of $u$.

 By Remark \ref{rem:sing-no-die} this is a contradiction, as by assumption $0 \in \sing(u_i)$ for every $i$.
\end{proof}

Notice, though we define (almost-)symmetry in terms of $u$, in the strata we consider only points in the free-boundary.  However, almost-symmetries of $u$ will propogate to almost-symmetries of the free-boundary.  Precisely, we have the following
\begin{proposition}\label{prop:fb-almost-sym}
There is an $\eps(\eps', \Lambda, n)$ so that if $u$ solves $(\star_{B_2(0), Q})$ with $|Q + 1/Q|_{C^0(B_2(0))} \leq \Lambda$, and $0 \in \partial \fb$, and $u$ is $(k, \eps)$-symmetric in $B_1$, then $\partial \fb$ is $(k, \eps')$-symmetric in $B_1(0)$ in the following sense:

There is a cone $C^{n-k-1}$ and an affine $k$-space $L^k$ so that
\begin{equation}\label{eqn:fb-almost-sym}
d_H( B_1(0) \cap \partial \fb, B_1(0)\cap (C^{n-k-1} \times L^k)) \leq \eps.
\end{equation}
\end{proposition}

\begin{proof}
Suppose the Proposition is false.  Then we have a sequence $u_i$ solving $(\star_{B_2(0), Q_i})$, and sequence $\eps_i$, which for each $i$ satisfy the hypothesis, but fail \eqref{eqn:fb-almost-sym}.  Then passing to a subsequence, we can assume $u_i \to u$ as in Theorem \ref{thm:compactness}.

Since $u_i \to u$ in $L^2(B_1(0))$, $u$ is $(k, 0)$-symmetric in $B_1(0)$, and hence $\partial \fb \cap B_1(0) = (C^{n-k-1} \times L^k) \cap B_1(0)$ for some cone $C^{n-k-1}$, and affine $k$-space $L^k$.  But since $\partial \{ u_i > 0 \} \to \partial \fb$ in $d_H$, we obtain a contradiction.
\end{proof}

\subsection{Main results}

We now state our main Theorems.  We have $k$-dimensional packing at all scales on the $(k,\eps)$-strata.
\begin{theorem}\label{thm:main-mass}
Suppose $u$ solves $(\star_{B_5(0), Q})$ with $|Q|_{C^\alpha(B_5(0))} \leq \Lambda$, $\min_{B_5(0)} Q \geq 1/\Lambda$, and $0 \in \partial \fb$.  Then for every $\eps > 0$, $0 < r \leq 1$, and $k \in \{1, \ldots, n-1\}$, we can find a collection of balls $\{B_r(x_i)\}_i$ which satisfy
\[
S^k_{\eps, r}(u)\cap B_1(0) \subset \cup_i B_r(x_i), \quad \#\{x_i\}_i \leq c(n, \Lambda, \eps, \alpha) r^{-k}
\]

In particular, we have
\[
|B_r(S^k_{\eps, r}(u) \cap B_1(0))| \leq c(n, \eps, \Lambda, \alpha) r^{n-k}  \quad \forall 0 < r \leq 1, 
\]
and 
\[
\haus^k(S^k_\eps(u) \cap B_1(0)) \leq c(n, \eps, \Lambda, \alpha).
\]
\end{theorem}

We have rectifiability of each $k$-stratum.
\begin{theorem}\label{thm:rectifiable}
Suppose $u$ solves $(\star_{\Omega, Q})$.  Then $S^k_\eps(u)$ is rectifiable for every $\eps$, and hence each stratum $S^k(u)$ is rectifiable.
\end{theorem}

Together these imply rectifiability and Minkowski estimates on the singular set.  Recall $k^*$ is the first dimension admitting a non-linear, $1$-homogenous solution to $(\star_{\R^{k^*}, 1})$.
\begin{corollary}\label{cor:main-sing}
Suppose $u$ solves $(\star_{B_5(0), Q})$ with $|Q|_{C^\alpha(B_5(0))} \leq \Lambda$, $\min_{B_5(0)} Q \geq 1/\Lambda$, and $0 \in \partial \fb$.

Then $\sing(u) \cap B_1(0)$ is rectifiable, and for every $0 < r \leq 1$, we have
\[
|B_r(\sing(u) \cap B_1(0))| \leq c(n,\Lambda, \alpha) r^{k^*} ,
\]
and
\[
(\haus^{n-1} \llcorner \partial \{u > 0\})(B_r(\sing(u) \cap B_1(0))) \leq c(n, \Lambda, \alpha) r^{k^* - 1}.
\]
In particular, 
\[
\haus^{n-k^*}(\sing(u) \cap B_1(0)) \leq c(n, \Lambda, \alpha).
\]
\end{corollary}

Moreover, since when $Q$ is constant, we have smoothness \emph{with estimates} at each regular scale, we can use the $S^{n-k^*}_{\eps, r}$ to obtain weak $L^{k^*}$ estimates on $D^2u$.  This is a quantitative version of Corollary \ref{cor:main-sing}, in which we interpret $\sing(u)$ as those points with $D^2 u = \infty$.
\begin{corollary}\label{cor:weak-Lp}
Suppose $u$ solves $(\star_{B_5(0), 1})$, with $0 \in \partial \fb$.  Then for every $0 < r \leq 1$, we have
\[
|B_r(\{ x \in \partial \fb \cap B_1(0) : |D^2 u(x)| > 1/r \} )| \leq c(n)  r^{k^*},
\]
and
\[
\haus^{n-1}(B_r(\{x \in \partial \fb \cap B_1(0) : |D^2u(x)| > 1/r \})) \leq c(n) r^{k^*-1}.
\]
\end{corollary}

\subsection{Other free-boundary problems}\label{section:other-probs}

The proof of Theorems \ref{thm:main-mass} and \ref{thm:rectifiable} require very little specific to the problem $(\star_{\Omega, Q})$.  In fact we only really require the compactness of Theorem \ref{thm:compactness}, the monotonicity formula of Theorem \ref{thm:weiss-monotonicity}, and some kind of a priori estimate like that of Theorem \ref{thm:altcaf}(3).  Corollary \ref{cor:main-sing} additionally requires an $\eps$-regularity theorem.  These properties are satisfied more-or-less verbatim for a variety of other free-boundary problems.  Let us detail some specific examples.

\textbf{Two-phase} Introduced by Alt-Caffarelli-Friedman in \cite{alt-caf-fried}, the problem requests $u$ minimize
\[
J^{d}(u) := \int_{\Omega} |Du|^2 + q_+^2(x) 1_{\{ u > 0 \}}(x) + q_-^2(x) 1_{\{ u < 0\}}(x) dx,
\]
for $q_+ \geq q_- > 0$ both in $C^\alpha(\Omega)$. Under the assumption that $q_+^2 -q_-^2 \geq c_0 > 0$,  Alt-Caffarelli-Friedman prove an analogue of Theorem \ref{thm:altcaf} (the $\eps$-regularity for $\partial \{u > 0\}$ in all dimensions follows from later work of Caffarelli see \cite{caffarelliharnack} and \cite{caffarelliharnack2}). The analogue of Theorem \ref{thm:compactness} for the two-phase problem is stated and proven in full in \cite{davidtoroalmostminimizers} (in the setting of almost-minimizers). Weiss \cite{weiss:partial-reg} proved almost-monotonicity (in the sense of Theorem \ref{thm:weiss-monotonicity}) of the quantity
\begin{align*}
W^{d}_r(y) &:= r^{-n} \int_{B_r(y)} |Du|^2 dx + q_+^2(y) r^{-n} |\{u > 0\}\cap B_r(y)| + q_-^2(y) r^{-n} |\{u < 0\}\cap B_r(y)| \\
&\quad - r^{-n-1} \int_{\partial B_r(y)} u^2 d\sigma.
\end{align*}

As such, we can apply our arguments to the free boundary, $\partial \{u > 0\}$ (note that the set $\partial \{u < 0\}$ is much less well understood without additional assumptions on $q_{\pm}$). 

\textbf{Vector-valued one-phase} The problem considered in \cite{caf-sha-yer} and \cite{maz-ter-vel} (see also \cite{kriv-lin}) is essentially asking to minimize
\[
J^{v}(U) := \int_{\Omega} |DU|^2 + Q^2(x) 1_{\{ |U| > 0 \}}(x) dx,
\]
where now $U : \Omega \to \R^N$ is vector valued, with the restriction that $U^i \geq 0$ on $\Omega$ (actually something slightly weaker is true, see Remark 1.5 in \cite{maz-ter-vel}). As usual, we stipulate $Q$ is positive and $C^\alpha$.  In \cite{caf-sha-yer}, \cite{maz-ter-vel}, again the analogues to Theorems \ref{thm:altcaf},  \ref{thm:compactness} are proven as is almost-monotonicity of
\[
W^{v}_r(y) := r^{-n} \int_{B_r(y)} |DU|^2 dx + Q^2(y) r^{-n}|\{|U| > 0\} \cap B_r(y)| - r^{-n-1} \int_{\partial B_r(y)} |U|^2 d\sigma.
\]

For the sake of precision let us remark that in \cite{maz-ter-vel}, the $\eps$-regularity result is written in terms of the density of $W^{v}$ (as opposed to flatness), much like we do in Theorem \ref{prop:reg-by-W}. Let us also remark that the compactness results in \cite{caf-sha-yer} and  \cite{maz-ter-vel} are written for special types of limits (blow-ups in the former and what are referred to as {\it pseudo-blowups} in the latter). However, adapting the proof for pseudo-blowups to one for generic limits requires merely cosmetic changes. 

\textbf{Almost-minimizers in one-phase} In \cite{davidtoroalmostminimizers}, they considered $u : \Omega \to \R$ which is an almost-minimizer for $J$ in the following sense: if $B_r(x) \subset\subset \Omega$, and $v - u \in W^{1,2}_0(B_r(x))$, then
\[
\left( \int_{B_r(x)} |Du|^2 + Q^2(x) 1_{\fb}(x) \right) \leq (1 + \kappa r^\gamma ) \left( \int_{B_r(x)} |Dv|^2 + Q^2(x) 1_{\{ v > 0\}}(x) \right),
\]
for some fixed $\gamma \in (0, 1)$, and positive $Q \in C^\alpha(\Omega)$.  Except for $\eps$-regularity (and of course, the representation formula) analogues of Theorems \ref{thm:altcaf} and  \ref{thm:compactness} were proven in \cite{davidtoroalmostminimizers}. In \cite{davidengelsteintoro}, it is shown that the quantity in \eqref{eqn:weissmonodef} is almost-monotone and an $\eps$-regularity theorem of type Theorem \ref{prop:reg-by-W} holds (see Propositions 5.2, 6.2 and Theorem 7.1 there).

For each of the above problems we can define (effective) strata using the Weiss-type monotonicity formulas.  Lemma \ref{lem:sym-from-drop} and Theorems \ref{thm:dichotomy}, \ref{thm:L2-est} continue to hold, with the same proofs.  We have
\begin{theorem}
Theorems \ref{thm:main-mass}, \ref{thm:rectifiable} and Corollary \ref{cor:main-sing} hold for the problems ``two-phase,'' ``vector-valued one-phase'' and ``almost-minimizing one-phase'' listed above.
\end{theorem}

\section{Rectifiable Reifenberg and proof outline}

The arguments of Naber-Valtorta are an effective dimension reduction in the vein of Federer and Almgren.  Recall that in the original dimension reduction argument, for any $1$-homogenous $u$ the singular set $\sing(u)$ could always be contained in some $L^{n-k*}$-space.  Following \cite{naber-valtorta:harmonic}, this is made effective: if $u$ is ``almost'' $1$-homogenous, then $\sing(u)$ is ``almost'' contained in some affine $L^{n-k^*}$.

The fundamental insight of Naber-Valtorta is that the error in each ``almost'' can be controlled by the density drop between scales, and from monotonicity \emph{the density drops are summable across scales}.  At a crude level this means that as you progress to smaller scales the errors accumulate to something finite, but making this precise takes significant effort.  We recommend the introduction of \cite{naber-valtorta:harmonic} for a more detailed outline.

Broadly, the argument of \cite{naber-valtorta:harmonic} splits into two parts.  The first is the Discrete- and Rectifiable-Reifenberg Theorems, which are purely GMT results that use Jones' $\beta$-numbers (Jones introduced the $L^\infty$ analogue in \cite{jonesbeta}) to prove packing bounds and rectifiability.  We state them below, and in this paper we will use them as black boxes, but invite the reader to examine them in detail in \cite{naber-valtorta:harmonic}, \cite{miskiewicz}, or \cite{env} (where it is done for general measures).

\begin{theorem}[Discrete-Reifenberg \cite{naber-valtorta:harmonic}]\label{thm:discrete-reifenberg}
Let $\{B_{r_q}(q)\}_q$ be a collection of disjoint balls, with $q \in B_1(0)$ and $0 < r_q \leq 1$, and let $\mu$ be the packing measure $\mu = \sum_q r_q^k \delta_{q}$, where $\delta_{q}$ is the Dirac delta at $q$. 

There exists a constant $\delta_{DR} > 0$ (depending only on dimension), such that if
\[
\int_0^{2r} \int_{B_{r}(x)} \beta^k_{\mu, 2}(z, s)^2 d\mu(z) \frac{ds}{s} \leq \delta_{DR}(n) r^k \quad \forall x \in B_1(0), 0 < r \leq 1, 
\]
then
\[
\mu(B_1(0)) \equiv \sum_q r_q^k \leq C_{DR}(n).
\]
\end{theorem}

\begin{theorem}[Rectifiable-Reifenberg \cite{naber-valtorta:harmonic}]\label{thm:rect-reifenberg}
Let $\delta_{RR} > 0$ be a constant depending only on dimension and $S$ be a set, satisfying
\[
\int_0^{2r} \int_{B_{r}(x)} \beta^k_{\haus^k \llcorner S, 2}(z, s)^2 d(\haus^k \llcorner S)(z) \frac{ds}{s} \leq \delta_{RR}(n) r^k \quad \forall x \in B_1(0), 0 < r \leq 1.
\]

Then $S \cap B_1(0)$ is $k$-rectifiable, and satisfies
\[
\haus^k(S \cap B_1(0) \cap B_r(x)) \leq C_{RR}(n) r^k \quad \forall x \in B_1(0), 0 < r \leq 1.
\]
\end{theorem}
The $L^p$-Jones $\beta$-numbers (defined in \eqref{eqn:beta} above, originally introduced in \cite{DavidandSemmes}) have been used in many instances to gain geometric information on the support of $\mu$. In \cite{DavidandSemmes} it was shown that if $\mu = \mathcal H^k|_E$ and $\mu(B(x,r)\cap E) \simeq r^k$ then bounds like those in Theorems \ref{thm:discrete-reifenberg} and \ref{thm:rect-reifenberg} above imply that $E$ is ``uniformly rectifiable". Later in \cite{guytororeifenberg}, similar bounds to those above (along with additional constraints on the approximating planes $V^k_{\mu, 2}(x, r)$) were used to construct Reifenberg-type parameterizations of sets. More recently, the paper of \cite{azzam-tolsa}, showed that bounds on the $\beta$-numbers, along with additional constraints on the density of the measure $\mu$, imply rectifiability of the support of $\mu$. These are just some of the wide array of applications the $L^p$ $\beta$-numbers have found in geometric measure theory and harmonic analysis. For a more complete discussion on Theorems \ref{thm:discrete-reifenberg} and \ref{thm:rect-reifenberg}, and how they fit in with the previous work in harmonic analysis and geometric measure theory, we recommend the introduction of \cite{naber-valtorta:harmonic}.




In this paper we focus on the second part of the argument of \cite{naber-valtorta:harmonic}.  The starting point is an estimate relating $\beta$-numbers of $S^k_\eps$ to density drop (Theorem \ref{thm:L2-est}).  This says, for example, that if $\mu$ a finite measure supported in $S^k_\eps$, and $x \in S^k_\eps$, and $W_{8r}(x) - W_{\gamma r}(x) < \gamma$ (for some small $\gamma$), then
\begin{equation}\label{eqn:teaser-L2}
\beta^k_{\mu,2}(x, r)^2 \leq c r^{-k} \int_{B_r(x)} W_{8r}(z) - W_r(z) d\mu(z).
\end{equation}

The naive strategy might be the following: cover $S^k_\eps$ with a Vitali collection of balls $\{B_r(x_i)\}_i$, then apply \eqref{eqn:teaser-L2} and discrete Reifenberg on the associated packing measure to obtain Theorem \ref{thm:main-mass}.  Unfortunately, there are several technicalities.

First, to exploit summability we need a priori $\mu$-mass control.  This is not a big problem: since we have control on the number of overlaps, we can inductively assume mass bounds at lower scales, to prove mass bounds at the next.  This induction is implemented in Lemma \ref{lem:centered-packing}.  We point out that Lemma \ref{lem:centered-packing} crucially requires small density drops \emph{at the ball centers}.

Second, we need \emph{smallness} of the density drop in both \eqref{eqn:teaser-L2} and discrete Reifenberg.  We adjust the naive strategy as follows: we instead build a cover of $S^k_\eps$ by balls $\{B_{r_i}(x_i)\}_i$ which have a small \emph{but definite} amount of density drop.  So, for each $i$, if $r_i > r$ we have
\[
\sup_{B_{2r_i}(x)} W_{2r_i} \leq \sup_{B_2(0)} W_2  - \eta.
\]

In each of these ``big" $B_{r_i}(x_i)$'s, we can recursively create a new covering by balls of smaller radii.  Each time we recurse we eat a definite amount of density, and decrease the radii by some fixed factor.  By monotonicity the process must terminate in finitely-many steps (see Theorem \ref{thm:key-packing}).

However, the most subtle issue is reconciling these two conditions on the cover, of having small density drops at the balls centers, while admitting definite density drop on the entire ball.  A naive cover satisfying one will not satisfy the other.

We are saved by the following elementary restatement of effective dimension reducing: If the points of small density drop ``look $k$-dimensional,'' then $S^k_\eps$ sits near a $k$-plane.  In particular, we get the following dichotomy (Theorem \ref{thm:dichotomy}): either we have small density drop on $S^k_\eps$, or the high-density points ``look $(k-1)$-dimensional.''  In the first (``good'') case we can use Lemma \ref{lem:centered-packing} to obtain packing bounds.  In the second (``bad'') case, we can get (very good!) packing bounds for the silly reason that we are one-dimension off.

This dichotomy allows for a general ``good ball/bad ball'' stopping-time construction, wherein we alternate constructing good/bad ``trees'' to progress down in scale (we remark that this has parallels to the Corona constructions of David and Semmes \cite{DavidandSemmes}).  See Section \ref{section:main-const} for a more detailed exposition.  This top-down approach is different from the original bottom-up method of \cite{naber-valtorta:harmonic} and we hope it will be more familiar to those working in harmonic analysis and PDEs. It has also been implemented in \cite{naber-valtorta:approx-harmonic}, \cite{delellis:q-valued}, \cite{env}.

\section{Symmetry and strata}\label{sec:symmetryandstrata}

In this section we relate the drop in density to the effective strata.  First, let us show that small density drops imply almost-symmetry.
\begin{lemma}\label{lem:sym-from-drop}
Take $\delta > 0$, and $u$ a solution to $(\star_{B_2(0), Q})$ with $|Q + 1/Q|_{C^{0}(B_2(0))} \leq \Lambda$ and $0 \in \partial \fb$.  Then there is a $\gamma = \gamma(n, \Lambda, \delta)$ so that if $[Q]_{\alpha, B_2} \leq \gamma$, and
\[
W_1(Q, u, 0) - W_\gamma(Q, u, 0) \leq \gamma,
\]
then $u$ is $(0, \delta)$-symmetric in $B_1(0)$.
\end{lemma}

\begin{proof}
Suppose not: there are sequences $\gamma_i \to 0$ and $u_i$ solving $(\star_{B_2(0), Q_i})$, with $|Q_i + 1/Q_i| \leq \Lambda$ and $[Q_i]_{\alpha, B_2(0)} \leq \gamma_i$, such that
\[
W_1(Q_i, u_i, 0) - W_{\gamma_i}(Q_i, u_i, 0) \leq \gamma_i, 
\]
but $u_i$ is not $(0, \delta)$-symmetric in $B_1(0)$.

Passing to a subsequence, we have $Q_i \to Q \equiv Q_0$ a constant, and $u_i \to u$ as in Theorem \ref{thm:compactness}, with $u$ solving $(\star_{B_{3/2}(0), Q_0})$.  By upper-semi-continuity,
\[
W_1(Q_0, u, 0) - W_0(Q_0, u, 0) \leq \limsup_i (W_1(Q_i, u_i, 0) - W_{\gamma_i}(Q_i, u_i, 0)) = 0.
\]
Therefore by Theorem \ref{thm:weiss-monotonicity}, $W_r(Q_0, u, 0)$ is constant in $r \in (0, 1]$, and hence $u$ is $1$-homogenous in $B_1(0)$.


Since
\[
\int_{B_1(0)} |u_i - u|^2\ dx \to 0,
\]
we have that each $u_i$ is $(0, o(1))$-symmetric, a contradiction for $i$ sufficiently large.
\end{proof}

We now prove a crucial dichotomy, which says either we have small drop in the entire $S^k_\eps$, or the high-density points look $(k-1)$-dimensional.

\begin{theorem}[Key Dichotomy]\label{thm:dichotomy}
There is an $\eta_0(n, \Lambda, E_0, \eps, \rho, \gamma, \eta', \alpha) << \rho$ so that the following holds: let $u$ solve $(\star_{B_4(0), Q})$, with $|Q + 1/Q|_{C^0(B_4(0))} \leq \Lambda$ and $0 \in \partial \fb$.  Take $E \in [0, E_0]$ with $\sup_{x \in B_1(0)} W_2(Q, u, x) \leq E$.

If $\eta \leq \eta_0$, and $[Q]_{\alpha, B_4(0)} \leq \eta$, then at least one of two possibilities occurs:
\begin{enumerate}[label=(\roman*)]
\item \label{itm:smalldroponstrata} we have $W_{\gamma\rho}(u,x) \geq E - \eta'$ on $S^k_{\eps, \eta} \cap B_1(0)$,  or

\item\label{itm:closetoplane} there is an affine $\ell^{k-1}$ so that $\{ W_{2\eta}(u,x) \geq E - \eta \} \cap B_1(0) \subset B_{\rho}(\ell)$.
\end{enumerate}
\end{theorem}

One should keep in mind that $\eta << \rho$.  Recall also that $E \leq c(n, \Lambda)$.  Theorem \ref{thm:dichotomy} will be an easy consequence of the following:

\begin{lemma}\label{lem:dichotomy}
There are $\eta_0(n, \Lambda, E_0, \eta', \gamma, \rho, \eps, \alpha) << \rho$, $\beta(n, \Lambda, E_0, \eta', \rho, \gamma, \eps, \alpha) < 1$ so that the following holds: Let $u$ solve $(\star_{B_4(0), Q})$, with $|Q + 1/Q|_{C^0(B_4(0))} \leq \Lambda$, $0 \in \partial\fb$, and $\sup_{x\in B_1(0)} W_2(Q, u, x) \leq E \in [0, E_0]$.

Suppose $\eta \leq \eta_0$, $[Q]_{\alpha, B_4(0)} \leq \eta$, and there are points $y_0, \ldots, y_k \in B_1(0)$ satisfying
\[
y_i \not\in B_\rho(<y_0, \ldots, y_{i-1}>), \quad \text{and} \quad W_{2\eta}(y_i) \geq E - \eta, \quad \forall i = 0,\ldots, k.
\]
then writing $L = <y_0, \ldots, y_k>$, we have
\begin{gather}\label{eqn:lem-dichotomy-1}
W_{\gamma\rho} \geq E - \eta' \text{ on } B_{\beta}(L)\cap B_1(0)
\end{gather}
and
\begin{gather}\label{eqn:lem-dichotomy-2}
S^k_{\eps, \eta} \cap B_1(0) \subset B_{\beta}(L).
\end{gather}
\end{lemma}

\begin{proof}[Proof of Lemma \ref{lem:dichotomy}]
First, towards a contradiction, suppose \eqref{eqn:lem-dichotomy-1} fails.  Then we have a sequence $u_j$ solving $(\star_{B_4(0), Q_j})$, and collections $E_j$, $y_{ij}$, $L^k_j$, $\eta_j$, and $\beta_j$, which satisfy the hypotheses with $[Q_j]_{\alpha, B_4(0)} \leq \eta_j \to 0$, $\beta_j \to 0$, but for each $j$ fail \eqref{eqn:lem-dichotomy-1} at some $x_j \in B_{\beta_j}(L_j)\cap B_1(0)$.

Passing to a subsequence, we can assume
\[
u_j \to u \text{ as in Theorem \ref{thm:compactness}}, \quad Q_j \to Q \equiv Q_0 \text{ a constant},
\]
and
\[
E_j \to E , \quad y_{ij} \to y_i, \quad L_j \to L, \quad x_j \to x \in \overline{B_1(0)} \cap L.
\]
Since $\rho$ is fixed, the $y_i$ span $L$.  The limit $u$ is a solution to $(\star_{B_3(0), Q_0})$.

By continuity of density, $\sup_{z \in B_1(0)} W_2(Q_0, u, z) \leq E$ and $W_{\gamma \rho}(Q_0, u, x) \geq E - \eta'$ at $x \in L \cap \overline{B_1(0)}$.  By upper-semi-continuity we know $W_0(Q_0, u, y_i) \geq E$.  Therefore $u$ is $1$-homogenous in at $y_i$ in $B_2(y_i)$ for each $i$, and so $u$ is independent of $L$ in $B_{1+\delta}(0) \subset \cup_i B_2(y_i)$ for some $\delta > 0$ (depending on the arrangement of the $y_i$'s).  In particular, we must have $W_0(Q_0, u, x) = E > W_{\gamma \rho}(Q_0, u, x)$, contradicting the sharp monotonicity.


We now suppose, again towards a contradiction, that \eqref{eqn:lem-dichotomy-2} fails.  We are allowed to shrink $\eta$, as this will only strengthen our hypothesis, and weaken conclusion \eqref{eqn:lem-dichotomy-1}.  We fix $\beta$, however.  We then have a sequences $u_j$, $E_j$, $y_{ij}$, $L_j$, $\eta_j$, which satisfy the hypotheses of Lemma \ref{lem:dichotomy} with $\eta_j \to 0$, but for each $j$ fail \eqref{eqn:lem-dichotomy-2}: there is some $x_j \in S^k_{\eps, \eta_j} \cap B_1(0) \setminus B_\beta(L_j)$.

We can assume $u_j$, $y_{ij}$, $L_j$, $x_j$, $E_j$ converge as before.  By the same argument as above, for some $\delta > 0$ the resulting $u$ will be $k$-symmetric with respect to $L$ in $B_{1+\delta}(0)$.  Since $x \in \overline{B_1(0)} \setminus B_\beta(L)$, any blow-up of $u$ at $x$ will be $(k+1)$-symmetric.  In particular, for some fixed $r > 0$, $u_j$ will be $(k+1, \eps)$-symmetric in $B_r(x)$.  This is a contradiction.
\end{proof}


\begin{proof}[Proof of Theorem \ref{thm:dichotomy}]
If we can find $y_0, \ldots, y_k$ as in Lemma \ref{lem:dichotomy}, then conclusion Theorem \ref{thm:dichotomy}\ref{itm:smalldroponstrata} is immediate.  Otherwise, failing to find the $y_i$s implies there is some $(k-1)$-space $\ell^{k-1}$, so that
\[
y \in B_1(0) \setminus B_\rho(\ell) \implies W_{2\eta}(y) < E - \eta,
\]
which is conclusion Theorem \ref{thm:dichotomy}\ref{itm:closetoplane}.
\end{proof}

\section{Global packing estimate}



Our strategy is the following: we cook up a covering of $S^k_\eps$ with balls of either small radius, or with a definite amount of density drop, and demonstrate a packing estimate on these balls.  The balls with small radius give us the right packing estimate.  The other balls do not, but the density drop means we can recurse inside.  Each time we recurse we drop a fixed amount of density, and so errors will only accumulate for a finite number of steps.


The Key Packing Estimate is the following.  Recall $c_0(n,\alpha)$ is the constant from Theorem \ref{thm:weiss-monotonicity}.
\begin{theorem}[Key Packing Estimate]\label{thm:key-packing}
There is an $\eta(n, \Lambda, \eps, \alpha)$ so that the following holds: Let $u$ be a solution to $(\star_{B_5(0), Q})$ with $|Q + 1/Q|_{C^0(B_5(0))} \leq \Lambda$, $[Q]_{\alpha, B_5(0)} \leq \frac{\eta}{2c_0}$, and $0 \in \partial \fb$.  Write $E = \sup_{B_{2}(0)} W_2$.

Given $0 < R \leq 1$, there is a collection of balls $\{B_{r_x}(x)\}_{x \in \cU}$, with $x \in B_1(0)\cap S^k_{\eps, \eta R}$ and $r_x \leq 1/10$, which satisfy the following properties:
\begin{enumerate}[label=(\Alph*)]
\item\label{itm:ucovering} Covering:
\[
S^k_{\eps, \eta R} \cap B_1(0) \subset \bigcup_{x \in \cU} B_{r_x}(x) .
\]

\item\label{itm:upacking} Packing:
\[
\sum_{x \in \cU} r_x^k \leq c(n, \Lambda, \eps, \alpha).
\]

\item\label{itm:uenergydrop} Energy drop: for every $x \in \cU$ we have $r_x \geq R$, and either
\[
r_x = R \quad \text{or} \quad \sup_{B_{2r_x}(x)} W_{2r_x} \leq E - \eta/2.
\]
\end{enumerate}
\end{theorem}

Sections \ref{sec:thel2estimate}, \ref{section:centered-packing}, and \ref{section:main-const} are devoted to proving Theorem \ref{thm:key-packing}.  Let us first see how this proves Theorem \ref{thm:main-mass}.

\begin{proof}[Proof of Theorem \ref{thm:main-mass} given Theorem \ref{thm:key-packing}]
We know $E \leq c(n, \Lambda)$ from equation \eqref{eqn:wcontrolled}.  Ensuring $c(n, \Lambda, \eps, \alpha)$ is sufficiently large, it suffices to prove Theorem \ref{thm:main-mass} for $r < \eta$.  Further, by scaling we can reduce to the case $[Q]_{\alpha, B_5(0)} \leq \eta/2c_0$.  Let us detail this.  For the duration of the proof $c$ denotes a generic constant depending only on $(n, \Lambda, \eps, \alpha)$.

Suppose we can show Theorem \ref{thm:main-mass} with $[Q]_\alpha \leq \eta/2c_0$.  Now for a general bound $[Q]_\alpha \leq \Lambda$, choose a Vitali cover $\{B_\rho(x_i)\}_i$ of $S^k_{\eta, R} \cap B_1(0)$ (with the $\rho/5$-balls disjoint) with $\rho$ chosen so that
\[
(5\rho)^\alpha \Lambda = \frac{\eta}{2c_0}.
\]
We can assume $R < \rho$, by ensuring $c$ is sufficiently large.  The dilated solution $u_{x_i, 4\rho}$ satisfies $(\star_{B_5(0), \hat Q})$ with $[\hat Q]_{\alpha, B_5(0)} \leq \eta/2c_0$.  Therefore by our hypothesis we have the mass bound
\[
R^{k-n} |B_R(S^k_{\eps, R}) \cap B_\rho(x_i)| \leq c \rho^k.
\]
Since there are at most $c(n)\rho^{-n}$ balls $\{B_\rho(x_i)\}_i$, we deduce the scale-$1$ bound
\[
R^{k-n} |B_R(S^k_{\eps, R}) \cap B_1(0)| \leq c (2c_0 \Lambda/\eta)^{(n-k)/\alpha} .
\]

Therefore we shall assume $[Q]_{\alpha, B_5(0)} \leq \eta/2c_0$.  Use Theorem \ref{thm:key-packing} to build the covering $\cU_1$ in $B_1(0)$.  If every $r_x = R$, then the packing and covering estimates of Theorem \ref{thm:key-packing} \ref{itm:ucovering}, \ref{itm:upacking} imply Theorem \ref{thm:main-mass} directly:
\begin{gather}\label{eqn:minkowski-proof}
R^{k-n} |B_R(S^k_{\eps, \eta R}) \cap B_1(0)| \leq \omega_n R^{k-n} \sum_{x \in \cU} (2R)^n = \omega_n 2^n \sum_{x \in \cU} r_x^k \leq c.
\end{gather}

Otherwise, if some $r_x > R$, let use Theorem \ref{thm:key-packing} to build a (finite!) sequence of refinements $\cU_1, \cU_2, \cU_3, \ldots$, which satisfy for each $i$ the following properties:
\begin{enumerate}
\item[($\mathrm A_i$)] covering:
\[
S^k_{\eps, \eta R} \cap B_1(0) \subset \bigcup_{x \in \cU_i} B_{r_x}(x), 
\]

\item[($\mathrm B_i$)] packing:
\[
\sum_{x \in \cU_i} r_x^k \leq c  \left( 1 + \sum_{x \in\cU_{i-1}} r_x^k \right) ,
\]

\item[($\mathrm C_i$)] energy drop: for every $x \in \cU_i$, we have $r_x \geq r$, and either
\[
r_x = r\quad\text{or}\quad \sup_{B_{2r_x}(x)} W_{2r_x} \leq E - i \eta/2.
\]

\item[($\mathrm D_i$)] radius control:
\[
\sup_{x \in \cU_i} r_x \leq 10^{-i} \quad \text{ and } \quad \cU_i \subset B_{1 + 10^{-i}}(0) \cap S^k_{\eps, \eta R}.
\]
\end{enumerate}

For any $r \leq 1$ and $x \in B_2(0) \cap \partial \fb$, we have from Theorem \ref{thm:weiss-monotonicity} that 
\[
W_r(x) \geq W_0(x) - c_0 [Q]_{\alpha, B_5(0)} \geq -\eta/2.
\]
Therefore, once $i \geq 2 + 2 E/\eta$, then every $x \in \cU_i$ will necessarily satisfy $r_x = R$.  Then, as in \eqref{eqn:minkowski-proof}, we obtain Theorem \ref{thm:main-mass} with a bound like
\[
R^{k-n}|B_R(S^k_{\eps, \eta R}) \cap B_1(0)| \leq c^{2 + 2E/\eta}.
\]

We have already constructed $\cU_1$, which satisfies $(A_1), (B_1), (C_1), (D_1)$ from Theorem \ref{thm:key-packing}.  Suppose, inductively, we have constructed $\cU_{i-1}$ satisfying properties $(A_{i-1})$, $(B_{i-1})$, $(C_{i-1}), (D_{i-1})$.

For each $x \in \cU_{i-1}$ with $r_x > R$, we wish to apply Theorem \ref{thm:key-packing} at scale $B_{r_x}(x)$ to obtain a new collection $\cU_{i, x}$.  Since $x \in B_{1+9^{-1}}(0)$, and $r_x < 1/10$, from how $Q$ scales we see that $u$, $Q$ continue satisfy the hypotheses of Theorem \ref{thm:key-packing} at scale $B_{r_x}(x)$.

However, from $(C_{i-1})$ we now have $\sup_{B_{2r_x}(x)} W_{2r_x} \leq E - (i-1) \eta/2$.  Therefore Theorem \ref{thm:key-packing}\ref{itm:uenergydrop} for the covering $\cU_{i, x}$ becomes
\begin{gather}\label{eqn:U-drop}
\sup_{B_{2r_y}(y)} W_{2r_y} \leq E - i \eta/2 \quad \forall y \in \cU_{i, x} \text{ with } r_y > R.
\end{gather}
Theorem \ref{thm:key-packing}\ref{itm:ucovering} is of course
\begin{gather}\label{eqn:U-covering}
S^k_{\eps, \eta R} \cap B_{r_x}(x) \subset \bigcup_{y \in \cU_{i, x}} B_{r_y}(y),
\end{gather}
and Theorem \ref{thm:key-packing}\ref{itm:upacking} becomes
\begin{gather}\label{eqn:U-packing}
\sum_{y \in \cU_{i, x}} r_y^k \leq c(n,\Lambda, \eps, \alpha) r_x^k.
\end{gather}
Moreover, from the construction of Theorem \ref{thm:key-packing} and $(D_{i-1})$ we have
\begin{equation}\label{eqn:U-radius}
\sup_{y \in \cU_{i, x}} r_y \leq 10^{-1} r_x \leq 10^{-i}.
\end{equation}

We then set
\[
\cU_i = \{ x \in \cU_{i-1} : r_x = R \} \cup \bigcup_{\{x \in \cU_{i-1} \,\, : \,\, r_x > R\}} \cU_{i, x}.
\]
From \eqref{eqn:U-drop}, \eqref{eqn:U-covering}, \eqref{eqn:U-packing}, and \eqref{eqn:U-radius}, the new $\cU_i$ satisfies inductive hypotheses $(A_i), (B_i), (C_i), (D_i)$.  This completes the construction of the covering refinements, and finishes the proof of Theorem \ref{thm:main-mass}.
\end{proof}

\section{The $L^2$-estimate}\label{sec:thel2estimate}


We prove the effective version of: ``for a $1$-homogenous $u$, the $S^k_\eps(u)$ is contained in some $k$-plane.''  Precisely, we show the $\beta$-numbers for $S^k_\eps$ are controlled by the density drop, whenever we are almost $1$-homogenous.

\begin{theorem}\label{thm:L2-est}
There is a $\delta(n, \Lambda, \eps, \alpha)$, so that the following holds:  Let $u$ be a solution to $(\star_{B_{10 r}(x), Q})$, with $|Q + 1/Q|_{C^0(B_{10r}(x))} \leq \Lambda$, $[Q]_{\alpha, B_{10r}(x)} \leq \delta$, and $x \in \partial \fb$.  Suppose
\begin{gather}\label{eqn:L2-est-hyp}
\left\{ \begin{array}{l} \text{$u$ is $(0, \delta)$-symmetric in $B_{8r}(x)$} \\
 \text{$u$ is \emph{not} $(k+1, \eps)$-symmetric in $B_{8r}(x)$}. \end{array}\right.
\end{gather}
Then for any finite Borel measure $\mu$, we have
\begin{gather}\label{eqn:L2-est-concl}
\beta^k_{\mu, 2}(x, r)^2 \leq \frac{c(\Lambda, \eps, n,\alpha)}{r^k} \int_{B_r(x)} W_{8r}(y) - W_r(y)+ c_0 [Q]_{\alpha, B_{10 r}(x)} (8r)^\alpha d\mu(y).
\end{gather}
\end{theorem}

\begin{remark}
Here is a baby case illustrating why Theorem \ref{thm:L2-est} should be true.  Suppose for simplicity $r = 1$, $x = 0$, $Q$ is constant, and the RHS is zero.  Then at each point $y \in \spt\mu$, $u$ must be $0$-symmetric in $A_{1,8}(y)$.  In particular, if there are $k+1$ linearly independent points in $\spt\mu \cap B_1(0)$, then $u$ must be $(k+1)$-symmetric in $A_{3,4}(0)$.  But then by $(0, \delta)$-symmetry of $B_8$, we have that $B_8$ is $(k+1, \eps_2(\delta))$-symmetric, with $\eps_2 \to 0$ as $\delta \to 0$.  Ensuring $\delta$ is sufficiently small gives a contradiction, and therefore $\spt\mu \cap B_1(0) \subset$ $k$-plane.
\end{remark}

\begin{proof}
Let $X$ be the $\mu$ center of mass of $B_r(x)$, and define the non-negative bilinear form
\[
Q(v, w) \equiv \fint_{B_r(x)} (v \cdot(x - X) )( w \cdot(x - X)) d\mu(x).
\]
Here $\cdot$ denotes the standard Euclidean inner product.  Let $v_1, \ldots, v_n$ be a orthonormal eigenbasis, and $\lambda_1 \geq \ldots \geq \lambda_n \geq 0$ the associated eigenvalues.  It's easy to check that
\[
V^k_{\mu,2}(x, r) = X + \mathrm{span}\{v_1, \ldots, v_k\}, \quad \beta^k_{\mu, 2}(x, r)^2 = \frac{\mu(B_r(x))}{r^k} (\lambda_{k+1} + \ldots + \lambda_n).
\]

We first claim that, for any $i$, and any $z$, 
\begin{gather}\label{eqn:evalue-density-point}
\lambda_i (v_i \cdot Du(z))^2 \leq \fint_{B_r(x)} |u(z) - (z - y) \cdot Du(z)|^2 d\mu(y).
\end{gather}
We calculate, using the definition of center of mass, 
\begin{align*}
\lambda_i (v_i \cdot Du(z)) 
&= Q(v_i, Du(z)) \\
&= \fint_{B_r(x)} (v_i \cdot( x - X))( Du(z) \cdot( x - X)) d\mu(y) \\
&= \fint (v_i \cdot(x - X))( u(z) - (z - x)\cdot Du(z)) d\mu(y) \\
&\leq \lambda_i^{1/2} \left( \fint |u(z) - (z - y)\cdot Du(z)|^2  d\mu(y) \right)^{1/2}.
\end{align*}
This proves \eqref{eqn:evalue-density-point}.

Writing $A_{3r, 4r}(x)$ for the annulus $B_{4r}(x) \setminus B_{3r}(x)$, we calculate
\begin{align*}
&\lambda_i r^{-n-2} \int_{A_{3r,4r}(x)} (v_i \cdot Du(z))^2 dz \\
&\leq r^{-n-2} \int_{A_{3r,4r}(x)} \fint_{B_r(x)} |u(z) - (z - y)\cdot Du(z)|^2 d\mu(y) dz \\
&\leq 5^n \fint_{B_r(x)} \int_{A_{3r,4r}(x)} |u(z) - (z - y)\cdot Du(z)|^2 |z - y|^{-n-2} dz d\mu(y) \\
&\leq 5^n \fint_{B_r(x)} \int_{A_{r,8r}(y)} |u(z) - (z - y) \cdot Du(y)|^2 |z - y|^{-n-2} dz d\mu(y) \\
&= c(n) \fint_{B_r(x)} W_{8r}(y) - W_r(y) + c_0 [Q]_\alpha (8r)^\alpha  d\mu(y) . 
\end{align*}

Up until now we haven't used hypothesis \eqref{eqn:L2-est-hyp}.  Our second claim is: ensuring $\delta(n, \Lambda, \eps, \alpha)$ is sufficiently small, then there exists some $c(n, \Lambda, \eps, \alpha)$ so that for \emph{any} orthonormal vectors $v_1, \ldots, v_{k+1}$, we have
\begin{gather}\label{eqn:lambdal2-lower-bound}
\frac{1}{c(n, \Lambda, \eps, \alpha)} \leq r^{-n-2} \int_{A_{3r, 4r}(x)} \sum_{i=1}^{k+1} (v_i \cdot Du(z))^2 dz.
\end{gather}
Notice \eqref{eqn:lambdal2-lower-bound} is scale-invariant: in our proof there is no loss in assuming $B_r(x) = B_1(0)$.  Suppose, towards a contradiction, \eqref{eqn:lambdal2-lower-bound} is false.  Then we have a sequence of $u_j$ solving $(\star_{B_{10}(0), Q_j})$ with $|Q_j + 1/Q_j|_{C^0} \leq \Lambda$ and $[Q_j]_\alpha \leq 1/j$, and orthonormal $v_{ij}$, so that
\[
\left\{ \begin{array}{l} \text{$u_j$ is $(0, 1/j)$-symmetric in $B_{8}(0)$} \\
 \text{$u_j$ is \emph{not} $(k+1, \eps)$-symmetric in $B_{8}(0)$} \end{array}\right.
\]
but
\begin{equation}\label{eqn:L2-counter-ineq}
1/j \geq \int_{A_{3, 4}(0)} \sum_{i=1}^{k+1} (v_{ij} \cdot Du_j(z))^2 dz .
\end{equation}

Passing to a subsequence, we can assume
\[
u_j \to u \text{ as in Theorem \ref{thm:compactness}}, \quad v_{ij} \to v_i,
\]
where $u$ solves $(\star_{B_{9}(0), Q})$ with $Q \equiv Q(0)$ constant, and $u$ is $0$-symmetric in $B_8(0)$.  From \eqref{eqn:L2-counter-ineq} we deduce $Du \cdot v_i \equiv 0$ in $A_{3,4}(0)$ for each $i = 1, \ldots, k+1$, and hence $u$ is $(k+1)$-symmetric in $B_8$.  But then the $u_j$ are $(k+1, o(1))$-symmetric in $B_8(0)$, a contradiction.  This proves \eqref{eqn:lambdal2-lower-bound}.

Combining \eqref{eqn:evalue-density-point} and \eqref{eqn:lambdal2-lower-bound} we deduce
\begin{align*}
\beta^k_{\mu, 2}(x, r)^2
&\leq \frac{\mu(B_r(x))}{r^k} n \lambda_{k+1} \\
&\leq \frac{\mu(B_r(x))}{r^k} n \cdot c(n, \Lambda, \eps,\alpha) \sum_{i=1}^{k+1} \frac{\lambda_i}{r^{n+2}} \int_{A_{3r, 4r}(x)} (v_i \cdot Du(z))^2 dz \\
&\leq \frac{c(n, \Lambda, \eps, \alpha)}{r^k} \int_{B_r(x)} W_{8r}(y) - W_r(y) + c_0 [Q]_\alpha (8r)^\alpha d\mu(y).
\end{align*}
This completes the proof of Theorem \ref{thm:L2-est}.
\end{proof}



\section{Centered density drop gives packing}\label{section:centered-packing}


We demonstrate how, if we have a collection of disjoint balls with small density drop \emph{at the centers}, then the $\beta$-estimate of Theorem \ref{thm:L2-est} and discrete Reifenberg Theorem \ref{thm:discrete-reifenberg} give good packing estimates.  The key idea is that the density drops are summable across scales (they are essentially a telescoping series), which allows us to ensure uniformly small $\beta$-number estimates.  The slight complication is that to sum \eqref{eqn:L2-est-concl} across scales we require packing bounds at lower scales.  We therefore must inductively prove packing scale-by-scale.


\begin{lemma}\label{lem:centered-packing}
There is an $\eta_1(n, \Lambda, \eps, \alpha)$ so that the following holds.  Take $\eta \leq \eta_1$, and let $u$ be a solution to $(\star_{B_5(0), Q})$ with $|Q + 1/Q|_{C^0(B_5(0))} \leq \Lambda$ and $[Q]_\alpha \leq \eta$.  Choose $R > 0$, and suppose $E \geq \sup_{B_1(0)} W_2$.

If $\{B_{2r_p}(p)\}_p$ is a collection of disjoint balls,satisfying
\begin{equation}\label{eqn:centered-balls}
W_{\eta r_p}(p) \geq E - \eta, \quad p \in S^k_{\eps, R} \cap B_1(0), \quad R \leq r_p \leq 1,
\end{equation}
then we have
\[
\sum_p r_p^k \leq c(n).
\]
\end{lemma}

\begin{proof}
Choose $\delta(n, \Lambda, \eps, \alpha)$ as in Theorem \ref{thm:L2-est}, and then $\gamma(n, \Lambda, \delta, \alpha)$ as in Lemma \ref{lem:sym-from-drop}.  Ensure
\[
\eta \leq \frac{\min\{\delta, \gamma\}}{2c_0 + 1}.
\]
Recall $c_0$ was the constant from Theorem \ref{thm:weiss-monotonicity}.  For convenience, in this proof we will write $r_i = 2^{-i}$.

For each integer $i \in \mathbb N$, define the packing measure
\begin{equation}\label{defofpackingmeasure}
\mu_i = \sum_{r_p \leq r_i} r_p^k \delta_p,
\end{equation}
and for shorthand write $\beta^k_i = \beta^k_{\mu_i, 2}$ (as defined for a general measure in \eqref{eqn:beta}).  Clearly the required estimate is equivalent to $\mu_0(B_1(0)) \leq c(n)$.

We make a few remarks about the $\beta_i$.  Suppose $x \in \spt \mu_i$, and $j \geq i$.  Then by disjointness
\begin{equation}\label{eqn:betaatdifferentscales}
\beta_i(x, r_j) = \left\{ \begin{array}{l l} \beta_j(x, r_j) & \text{ if $x \in \spt \mu_j$}, \\
0 & \text{ otherwise.} \end{array}\right. 
\end{equation}

On the other hand, since $W_{8r_i}(x) - W_{\eta r_i}(x) \leq (2c_0 + 1)\eta$, we have by Lemma \ref{lem:sym-from-drop}, Theorem \ref{thm:L2-est}, and our choice of $\eta$ that
\begin{equation}\label{betail2estimates}
\beta_i(x, r_i)^2 \leq c(n, \Lambda, \eps) r_i^{-k} \int_{B_{r_i}(x)} W_{8r_i}(y) - W_{ r_i}(y) + c_0 \eta (8r_i)^\alpha \, d\mu_i(y) ,
\end{equation}
whenever $r_i < 2^{-4}$.

For $r_i \leq 2^{-4}$, we shall inductively prove the estimate
\begin{equation}\label{eqn:inductivemassbound} \tag{$\dagger_i$}
 \mu_i(B_{r_i}(x)) \leq C_{DR}(n) r_i^k \quad \forall x \in B_1(0).
\end{equation}
Here $C_{DR}$ is the constant from Theorem \ref{thm:discrete-reifenberg}.  We observe that \eqref{eqn:inductivemassbound} vacuously holds for $i$ so large that $r_i < R$, as in this case $\mu_i \equiv 0$.  Let us suppose the inductive hypothesis that $(\dagger_j)$ holds for all $j \geq i+1$.

Fix an $x \in B_1(0)$.  By a packing argument and our inductive hypothesis we can suppose
\begin{equation}\label{massboundstimes4}
\mu_j(B_{4r_j}(x)) \leq \Gamma(n) r_j^k \quad \forall j \geq i-2, \forall x \in B_1(0),
\end{equation}
where $\Gamma = c(n) C_{DR}$.  We elaborate.  We have $\mu_j(B_{4r_j}(x)) = \mu_{j+2}(B_{4r_j}(x)) + \sum r_p^k$ where we sum over $p\in B_{4r_j}(x)$ having $r_{j+2} < r_p \leq r_{j}$. Since the  $B_{2r_p}(p)$ are disjoint we are summing over at most $c(n)$ points.

We calculate, using Fubini, 
\begin{align*}
&\sum_{r_j \leq 2r_i} \int_{B_{2r_i}(x)} \beta_i(z, r_j)^2 d\mu_i(z) \\
&\stackrel{\eqref{eqn:betaatdifferentscales}}{=} \sum_{r_j \leq 2r_i} \int_{B_{2r_i}(x)} \beta_j(z, r_j)^2 d\mu_j(z) \\
&\stackrel{\eqref{betail2estimates}}{\leq} c \sum_{r_j \leq 2r_i} \frac{1}{r_j^k} \int_{B_{2r_i}(x)} \int_{B_{r_j}(z)} W_{8r_j}(y) - W_{r_j}(y) + c(n,\alpha) \eta r_j^\alpha \, d\mu_j(y) d\mu_j(z) \\
&\leq c \sum_{r_j \leq 2r_i} \int_{B_{2r_i + r_j}(x)} \frac{\mu_j(B_{r_j}(y))}{r_j^k} \left( W_{8r_j}(y) - W_{r_j}(y) + c \eta r_j^\alpha \right) d\mu_j(y) \\
&\stackrel{\eqref{eqn:inductivemassbound}}{\leq} c \Gamma \int_{B_{4r_i}(x)} \sum_{r_j \leq 2r_i} W_{8r_j}(y) - W_{r_j}(y) + c \eta r_j^\alpha \, d\mu_j(y) \\
&\stackrel{\eqref{defofpackingmeasure}}{\leq} c \Gamma \left( \sum_{ p \in B_{4r_i}(x)\cap \spt\mu_i } r_p^k (W_{16r_i}(p) - W_{r_p}(p) + c  \eta ) \right) \\
&\leq c \Gamma \cdot \eta \cdot \mu_i(B_{4r_i}(x))  \\
&\stackrel{\eqref{massboundstimes4}}{\leq} c(n, \Lambda, \eps, \alpha) \Gamma^2 \eta r_i^k .
\end{align*}

Ensuring $\eta(n, \Lambda, \eps, \alpha)$ is sufficiently small, we deduce
\[
\sum_{r_j \leq 2r_i} \int_{B_{2r_i}(x)} \beta_i(z, r_j)^2 d\mu_i(z) \leq \delta_{dr} r_i^k ,
\]
and therefore Discrete-Reifenberg (Theorem \ref{thm:discrete-reifenberg}) implies
\[
\mu_i(B_{r_i}(x)) \leq C_{DR} r_i^k .
\]
This proves $\eqref{eqn:inductivemassbound}$, and therefore by mathematical induction $(\dagger)$ holds for all $r_i \leq 2^{-4}$.  The required bound on $\mu_0$ now follows by a simple packing argument.
\end{proof}


\section{A corona-type decomposition}\label{section:main-const}


In this section we build the cover of Theorem \ref{thm:key-packing}.  The complication is in reconciling condition \ref{itm:uenergydrop} of Theorem \ref{thm:key-packing}, requiring a \emph{definite} density drop on the entire balls, and condition \eqref{eqn:centered-balls} of Lemma \ref{lem:centered-packing}, requiring \emph{small} density drops at the centers.

The crucial observation that makes it work is the dichotomy Theorem \ref{thm:dichotomy}: in any ball either we have small density drop in the entire $S^k_{\eps, \eta R}$, or the high-density points are concentrated near a lower-dimensional $(k-1)$-plane.  In other words, whenever we cannot use Lemma \ref{lem:centered-packing}, we get a small enough $k$-dimensional packing estimate on $S^k_{\eps, \eta R}$ to compensate for naive overlaps.

As suggested in the Introduction, we call balls satisfying the first condition of having small drops \emph{good}, and balls satisfying the second condition of looking $(k-1)$-dimensional \emph{bad} (see Definition \ref{def:balls}).  We shall implement two different stopping-time arguments, one for good balls (Section \ref{section:good-tree}) and one for bad balls (Section \ref{section:bad-tree}).  In each case we build a \emph{tree} of good or bad balls, which is a sequence of coverings at smaller and smaller scales by good balls, bad balls, and balls satisfying the stopping condition \ref{itm:uenergydrop} of Theorem \ref{thm:key-packing}.  We then chain these trees together (Section \ref{section:alt-trees}) to obtain our estimate.  Let us detail some the tree constructions.

\vspace{2.5mm}
A \emph{good tree} is built in the following way.  Start at some initial good ball $B_{r_0}(g_0)$.  By virtue of being good, we have small drop in $S^k_{\eps, \eta R} \cap B_{r_0}(g_0)$ down to a very small scale.  We let the good/bad balls at scale $r_1$ be a Vitali cover of $S^k_{\eps, \eta R} \cap B_{r_0}(g_0)$.  So $S^k_{\eps, \eta R} \cap B_{r_0}(g_0)$ is covered by the good $r_1$-balls, and bad $r_1$-balls.

Now we define the good/bad balls at scale $r_2$ to be a Vitali cover of
\[
S^k_{\eps, \eta R} \cap B_{r_0}(g_0) \cap \text{(good $r_1$-balls)} \setminus \text{(bad $r_1$-balls)}.
\]
So $S^k_{\eps, \eta R}$ is covered by the good $r_2$-balls, and bad balls at scales $r_1$ and $r_2$.  We continue in this fashion, inducting into the good $r_i$-balls, and avoiding bad balls at scales $r_i, r_{i-1}, \ldots, r_1$.  We continue until we hit $r_i = R$, and we end up with a cover of $S^k_{\eps, \eta R} \cap B_{r_0}(g_0)$ with stop balls at scale $R$, and bad balls at scale $r_1, \ldots, r_i, \ldots, R$. 

Each stop/bad ball center has small density drop, by virtue of living inside a bigger good ball.  We can use Lemma \ref{lem:centered-packing} to obtain packing on the resulting cover.  See Theorem \ref{thm:good-tree}.

\vspace{2.5mm}
A \emph{bad tree} is built as follows.  Start at some initial bad ball $B_{r_0}(b_0)$.  By definition, there is some $(k-1)$-plane $\ell_0^{k-1}$, so that the points of high density are clustered near $\ell_0$.

We define the good/bad balls at scale $r_1$ to be a Vitali cover of $S^k_{\eps, \eta R} \cap B_{r_0}(B_0) \cap B_{2r_1}(\ell_0)$, and define the stop balls at scale $r_1$ to be a Vitali cover of $S^k_{\eps, \eta R} \cap (B_{r_0}(b_0) \setminus B_{2r_1}(\ell_0))$.  So $S^k_{\eps, \eta R} \cap B_{r_0}(b_0)$ is covered by the good/bad/stop balls at scale $r_1$, and by construction each stop ball must have uniformly large density drop.

We define the good/bad $r_2$-balls in the same way, covering $(k-1)$-planar neighborhoods in each bad $r_1$-ball (we don't need to avoid previous balls).  The stop $r_2$-balls cover the complements of $(k-1)$-planar neighborhoods in bad $r_1$-balls.  We continue until $r_i = R$, and end up with a cover of $S^k_{\eps, \eta R} \cap B_{r_0}(b_0)$ with stop balls at scale $R$, and good/stop balls at scales $r_1, \ldots, r_i, \ldots, R$, with the property that stop balls at scale $> R$ have large density drop.

At each scale the bad balls cover only a $(k-1)$-dimensional region, and so by choosing our scale-drop sufficiently small we can obtain $k$-packing estimates on \emph{all} bad balls across all scales.  This in turn gives a $k$-packing estimate on all the good/stop balls.  In fact we can make the $k$-packing estimate \emph{very small}.  See Theorem \ref{thm:bad-tree}.

\vspace{2.5mm}

Each good/bad tree satisfies the required decomposition of Theorem \ref{thm:key-packing} away from the bad/good balls.  In a good tree we may refer to bad balls as the \emph{tree leaves}, and similar the good balls are \emph{bad tree leaves}.  Given an initial good/bad tree, rooted at $B_1(0)$, let us build secondary bad/good trees in all the leaves.  In the leaves of each secondary tree, build tertiary trees.  Continuing in this fashion, we obtain a sequence of decompositions with smaller and smaller leaves.  Eventually, all the balls will satisfy the stopping conditions of Theorem \ref{thm:key-packing} \ref{itm:uenergydrop}.

Each time we build a new family of trees, the trees switch type.  This is very important, as each time we build a new tree we incur double-counting errors, because we essentially forget all the other trees exist.  The type-switching means we can kill the double-counting errors with the small bad-tree packing.  See Theorem \ref{thm:all-trees}

\vspace{2.5mm}

For the duration of this section we assume the hypotheses of Theorem \ref{thm:key-packing}.  So, $u$ solves $(\star_{B_5(0), Q})$, with $|Q + 1/Q|_{C^0(B_5(0))} \leq \Lambda$ and $[Q]_{\alpha, B_5(0)} \leq \frac{\eta}{2c_0}$, for $\eta$ to be chosen below.  We fix $E = \sup_{B_2(0)} W_2$, and fix some $R \in (0, 1]$.

First, choose $\rho < 1/10$ so that 
\[
2c_1(n) c_2(n) \rho \leq 1/2 ,
\]
where $c_1$ as in Theorem \ref{thm:good-tree}, and $c_2$ as in Theorem \ref{thm:bad-tree}.  Let
\[
\gamma = \eta' = \eta_1(n, \Lambda, \eps, \alpha)/20
\]
as in Lemma \ref{lem:centered-packing}.  Now take
\[
\eta = \eta_0(n, \Lambda, E + 1, \eps, \eta', \gamma, \rho ,\eps, \alpha)
\]
as in Theorem \ref{thm:dichotomy}.  Throughout this section we adhere to the following convention:
\begin{definition}
Write $r_i = \rho^i$.
\end{definition}


Precisely, our notions of good and bad are as follows.
\begin{definition}\label{def:balls}
Take $x \in B_2(0)$, and $R < r < 2$.  We say the ball $B_r(x)$ is \emph{good} if 
\[
W_{\gamma \rho r} \geq E - \eta' \quad \text{ on } \quad S^k_{\eps, \eta R} \cap B_r(x),
\]
and we say $B_r(x)$ is \emph{bad} if it isn't good.

By Theorem \ref{thm:dichotomy} with $E + \eta/2$ in place of $E$ (which is admissible by monotonicity and our choice of $[Q]_{\alpha, B_5(0)} \leq \eta/2c_0$), in any bad ball $B_r(x)$ we have
\[
\{ W_{2\eta r} \geq E - \eta/2 \} \cap B_r(x) \subset B_{\rho r}(\ell^{k-1})
\]
for some affine $(k-1)$-plane $\ell^{k-1}$.
\end{definition}




\subsection{Good tree construction}\label{section:good-tree}

Suppose $B_{r_A}(a)$ is a good ball at scale $A \geq 0$, with $a \in B_1(0)$.  We define precisely the good tree at $B_{r_A}(a)$.  As explained at the start of the Section, the good tree is a sequence of coverings at finer and finer scales, which will decompose $S^k_\eps \cap B_{r_A}(a)$ into a family of bad balls with packing estimates, and a Vitali collection of balls of radius $\approx R$.

We inductively define, for each scale $i \geq A$, a family of good balls $\{B_{r_i}(g)\}_{g \in \cG_i}$, bad balls $\{B_{r_i}(b)\}_{b \in \cB_i}$, and stop balls $\{B_{r_i}(s)\}_{s \in \cS_i}$.  At scale $i = A$, we let $\cG_A = \{a\}$ (so $B_{r_A}(a)$ is the only good ball), and $\cB_A = \cS_A = \emptyset$ (so no scale-$A$ bad or stop balls).

Suppose we have constructed the good/bad/stop balls down through scale $i-1$.  We let $\{z\}_{z \in J_i}$ be a maximal $2r_i/5$-net in
\[
B_1(0) \cap S^k_{\eps, \eta R} \cap B_{r_A}(a) \cap B_{r_{i-1}}(\cG_{i-1}) \setminus \bigcup_{\ell = A}^{i-1} B_{r_\ell}(\cB_\ell).
\]

If $r_i \leq R$, then we let $\cS_i = J_i$, and $\cG_i = \cB_i = \emptyset$.  In other words, we stop building the tree.  Otherwise, we differentiate the $z$'s into $\cG_i \cup \cB_i := J_i$ (disjoint union) by Definition \ref{def:balls}, and take $\cS_i = \emptyset$.  This completes the good tree construction.

\begin{definition}
The construction defined above is called the \emph{good tree rooted at $B_{r_A}(a)$}, and may be written as $\cT_G \equiv \cT_G(B_{r_A}(a))$.  Given such a good tree $\cT_G$, we define the \emph{tree leaves} $\cF(\cT_G) := \cup_i \cB_i$ to be the collection of all bad ball centers, across all scales.  Similarly we let $\cS(\cT_G) = \cup_i \cS_i$ be the collection of stop ball centers.

In a slight abuse of notation, we let $\{r_f\}_{f \in \cF(\cT_G)}$ and $\{r_s\}_{s \in \cS(\cT_G)}$ be the associated radius functions for the leaves $\cF(\cT_G)$, stop balls $\cS(\cT_G)$ (resp.), so e.g. if $s \in \cS_i \subset \cS(\cT_G)$, then $r_s = r_i$.
\end{definition}

We prove the following Theorem for good trees.
\begin{theorem}\label{thm:good-tree}
Let $\cT_G = \cT_G(B_{r_A}(a))$ be a good tree.  We have
\begin{enumerate}[label=(\Alph*)]
\item\label{itm:goodleafpacking} Tree-leaf packing:
\[
\sum_{f \in \mathcal{F}(\cT_G)} r_f^k \leq c_1(n) r_A^k.
\]

\item\label{itm:goodstoppacking} Stop ball packing:
\[
\sum_{s \in \cS(\cT_G)} r_s^k \leq c(n) r_A^k.
\]

\item\label{itm:goodcoveringcontrol} Covering control:
\[
B_1(0) \cap S^k_{\eps, \eta R} \cap B_{r_A}(a) \subset \bigcup_{s \in \cS(\cT_G)} B_{r_s}(s) \cup \bigcup_{f \in \cF(\cT_G)} B_{r_f}(f).
\]

\item\label{itm:goodstopstructure} Stop ball structure: for any $s \in \cS(\cT_G)$, we have $\rho R \leq r_s \leq R$.
\end{enumerate}
\end{theorem}

\begin{proof}
Let us point out the two key properties of bad and stop balls.  First, direct from construction, the collections of stop/bad $r_i/5$-balls
\begin{align*}
\{ B_{r_i/5}(b) : b \in \cup_{i=A}^\infty \cB_i \} \cup \{ B_{r_i/5}(s) : s \in \cup_{i=A}^\infty \cS_i \}
\end{align*}
are all pairwise disjoint, and centered in $S^k_{\eps, \eta R}$.

Second, since each stop/bad ball is also centered in a good ball at a previous scale, we have small density drop in the centers of every stop bad ball, i.e. for each $i$:
\[
W_{\gamma r_i}(b) \equiv W_{\gamma \rho r_{i-1}}(b) \geq E - \eta' \quad \forall b \in\cB_i, \quad\text{and}\quad W_{\gamma r_i}(s) \geq E - \eta' \quad \forall s \in\cS_i.
\]

Since by monotonicity we have $\sup_{B_{r_A}(a)} W_{2r_A} \leq E + \eta'$, we can use Lemma \ref{lem:centered-packing} at scale $B_{r_A}(a)$ to prove packing estimates \ref{itm:goodleafpacking}, \ref{itm:goodstoppacking}.

Conclusion \ref{itm:goodcoveringcontrol} is an elementary induction argument: for each $i \geq A$, we claim that
\begin{gather}\label{eqn:inductive-bad}
B_1(0) \cap S^k_{\eps, \eta R} \cap B_{r_A}(a) \subset B_{r_i}(\cG_i) \cup \bigcup_{\ell=0}^i B_{r_\ell}(\cB_\ell \cup \cS_\ell).
\end{gather}
When $i = A$ \eqref{eqn:inductive-bad} trivially is true.  Suppose, by inductive hypothesis, that \eqref{eqn:inductive-bad} holds at $i-1$.  Then by construction we have
\[
B_1(0) \cap S^k_{\eps, \eta R} \cap B_{r_{i-1}}(\cG_{i-1}) \setminus \cup_{\ell=A}^{i-1} B_{r_\ell}(\cB_\ell) \subset B_{r_i}(\cB_i \cup \cG_i \cup \cS_i).
\]
This proves \eqref{eqn:inductive-bad} at stage $i$.  When $r_i \leq R$ there are no good balls, and therefore \eqref{eqn:inductive-bad} implies conclusion \ref{itm:goodcoveringcontrol}.

Conclusion \ref{itm:goodstopstructure} follows because the only $i$ for which $\cS_i \neq\emptyset$ is when $r_i \leq R$, in which case necessarily $r_{i-1} > R$.
\end{proof}


\subsection{Bad tree construction}\label{section:bad-tree}

Suppose $B_{r_A}(a)$ is a bad ball at scale $A \geq 0$, with $a \in B_1(0)$.  In the following we construct the \emph{bad tree at $B_{r_A}(a)$}, which decomposes $S^k_{\eps, \eta R} \cap B_{r_A}(a)$ into a collection of good balls and balls with definite energy drop, each with packing estimates, and a Vitali collection of balls of radius $\approx R$.

As before we inductively define, for each scale $i \geq A$, a family of good balls $\{B_{r_i}(g)\}_{g \in \cG_i}$, bad balls $\{B_{r_i}(b)\}_{b \in \cB_i}$, and stop balls $\{B_{\eta r_{i-1}}(s)\}_{s \in \cS_i}$.  At scale $i = 0$, we let $\cB_A = \{a\}$, and $\cG_A = \cS_A = \emptyset$.  However, let us emphasize that these good/bad/stop balls are \emph{distinct} from the tree construction in section \ref{section:good-tree}.  Moreover, notice we define these stop balls to have (the smaller) radius $\eta r_{i-1}$ instead of $\rho r_{i-1} \equiv r_i$.  This is to ensure a uniform density drop on big stop balls.

For each bad ball $b \in \cB_i$ we have a $(k-1)$-affine plane $\ell_b^{k-1}$, associated to Theorem \ref{thm:dichotomy}\ref{itm:closetoplane}.

Suppose we have constructed the good/bad/stop balls down through scale $i-1$.  If $r_i \leq R$, then take $\cG_i = \cB_i = \emptyset$, and $\cS_i$ to be a maximal $2\eta r_{i-1}/5$-net in
\[
B_1(0) \cap S^k_{\eps, \eta R} \cap B_{r_A}(a) \cap B_{r_{i-1}}(\cB_{i-1}).
\]
So, we're stopping the tree.  Remember $\eta << \rho$, so $\eta r_{i-1} < r_i \leq R$.

Otherwise, if $r_i > R$, we define $\cS_i$ to be a maximal $2\eta r_{i-1}/5$-net in
\[
B_1(0) \cap S^k_{\eps, \eta R} \cap B_{r_A}(a) \cap \bigcup_{b \in \cB_{i-1}} \left( B_{r_{i-1}}(b) \setminus B_{2\rho r_{i-1}}(\ell_b) \right),
\]
and we let $\{g\}_{g \in \cG_i} \cup \{b\}_{b \in \cB_i}$ be a maximal $2r_i/5$-net in
\[
B_1(0) \cap S^k_{\eps, \eta R} \cap B_{r_A}(a) \cap \bigcup_{b \in \cB_{i-1}} \left( B_{r_{i-1}}(b) \cap B_{2\rho r_{i-1}}(\ell_b) \right) .
\]
This completes the bad tree construction.

\begin{definition}
The construction defined above is called the \emph{bad tree rooted at $B_{r_A}(a)$}, and may be written as $\cT_B \equiv \cT_B(B_{r_A}(a))$.  Given such a bad tree, we define the tree leaves to be $\cF(\cT_B):= \cup_i \cG_i$, the collection of all good ball centers, and set $\cS(\cT_B) = \cup_i \cS_i$ be the collection of stop ball centers.

As before, we write $r_f$, $r_s$ for the associated radius function.  So, e.g., if $s \in \cS_i \subset \cS(\cT_B)$, then $r_s = \eta r_{i-1}$.
\end{definition}

We prove the following Theorem for bad trees:
\begin{theorem}\label{thm:bad-tree}
Let $\cT_B = \cT_B(B_{r_A}(a))$ be a bad tree.  Then we have:
\begin{enumerate}[label=(\Alph*)]
\item\label{itm:badleafpacking} tree-leaf packing, with \emph{small} constant:
\[
\sum_{f \in \mathcal{F}(\cT_B)} r_f^k \leq 2c_2(n) \rho r_A^k
\]

\item\label{itm:badstopballpacking} Stop ball packing:
\[
\sum_{s \in \cS(\cT_B)} r_s^k \leq c(n, \eta) r_A^k.
\]

\item\label{itm:badcoveringcontrol} Covering control:
\[
B_1(0) \cap S^k_{\eps, \eta R} \cap B_{r_A}(a) \subset \bigcup_{s \in \cS(\cT_B)} B_{r_s}(s) \cup \bigcup_{f \in \cF(\cT_B)} B_{r_f}(f).
\]

\item\label{itm:badstopballstructure} Stop ball structure: for any $s \in \cS(\cT_B)$, then we have
\[
\eta R \leq r_s \leq R, \quad \text{and/or} \quad \sup_{B_{2r_s}(s)} W_{2r_s} \leq E - \eta/2.
\]
\end{enumerate}

\end{theorem}

\begin{proof}
Take $r_i > R$.  The good/bad ball centers $\cG_i \cup \cB_i$ lie in $B_{2\rho r_{i-1}}(\ell^{k-1})$, and the $r_i/5$-balls are disjoint.  Therefore, given any bad-ball $B_{r_{i-1}}(b)$, we have
\[
\#\{(\cG_i \cup \cB_i ) \cap B_{r_{i-1}}(b)\} \leq \frac{\omega_{k-1} \omega_{n-k+1} (3\rho)^{n-k+1}}{\omega_n (\rho/5)^n} \leq c_2(n) \rho^{1-k} .
\]
We deduce
\[
\#\{\cG_i \cup \cB_i\} r_i^k \leq c_2 \rho \#\{\cB_{i-1}\} r_{i-1}^k \leq c_2\rho \#\{\cB_{i-1}\cup \cG_{i-1}\}r_{i-1}^k \leq \ldots \leq (c_2 \rho)^{i-A} r_A^k,
\]
and therefore
\[
\sum_{i=A+1}^\infty \#\{\cG_i \cup \cB_i\}r_i^k \leq \sum_{i=A+1}^\infty (c_2 \rho)^{i-A}r_A^k \leq 2c_2\rho r_A^k,
\]
by our choice of $c_2\rho < 1/2$.

Since every good leaf is of scale $\leq r_{A+1}$ and $> R$, this proves the packing estimate \ref{itm:badleafpacking}.  It will also imply estimate \ref{itm:badstopballpacking} as follows.

Given $i \geq A+1$, the stop balls $\{B_{\eta r_{i-1}}(s)\}_{ s \in \cS_i}$ form a Vitali collection centered in $B_{r_{i-1}}(\cB_{i-1})$.  This implies that
\[
\#\{\cS_i\} \leq \frac{10^n}{\eta^n} \#\{\cB_{i-1}\}.
\]
Of course there aren't any stop balls at scale $r_A$.  We deduce
\[
\sum_{i=A+1}^\infty \#\{\cS_i\} (\eta r_{i-1})^k \leq 10^n \eta^{k-n} \sum_{i=A}^\infty \#\{\cB_i\} r_i^k \leq c(n, \eta) r_A^k.
\]
This proves estimate \ref{itm:badstopballpacking}.

Conclusion \ref{itm:badcoveringcontrol} follows precisely as in Theorem \ref{thm:good-tree}.  We prove conclusion \ref{itm:badstopballstructure}.  Take a stop ball center $s \in \cS_i$.  First suppose $r_i > R$.  Then necessarily $s \in B_{r_{i-1}}(b) \setminus B_{2\rho r_{i-1}}(b)$ for some bad ball $b \in \cB_{i-1}$.  By Theorem \ref{thm:dichotomy} and Definition \ref{def:balls}, and our choice $\eta < \rho/2$, we have
\[
\sup_{B_{2r_s}(s)} W_{2r_s} \leq \sup_{B_{\rho r_{i-1}}(s)} W_{2\eta r_{i-1}} \leq E - \eta/2.
\]
Conversely, the only way $r_i \leq R$ could occur is if $r_{i-1} \geq R$.  In this case we have
\[
R \geq \rho r_{i-1} \geq \eta r_{i-1} = r_s \geq \eta R.
\]
\end{proof}


\subsection{Alternating the trees}\label{section:alt-trees}

Our aim is to build a covering as in Theorem \ref{thm:key-packing}.  In any given tree, the stop balls and leaves cover $S^k_{\eps, \eta R} \cap B_1(0)$, and the stop balls satisfy Theorem \ref{thm:key-packing}\ref{itm:upacking} and \ref{itm:uenergydrop}, but the leaves do not.  We therefore implement the following strategy: build first a tree from $B_1(0)$ (let's say it's a good tree); at any bad leaf of our good tree, build a bad tree; in any good leaf of this collection of bad trees, build a good tree; etc.

Each time we build a new tree we switch type, and we can therefore use the small packing of bad trees to cancel overlap errors incurred by tree switching.  Moreover, each time we alternate tree-type our balls shrink by at least $\rho$, so the process must terminate at a collection of stop balls, which cover $S^k_{\eps, \eta R} \cap B_1(0)$ and satisfy the properties of Theorem \ref{thm:key-packing}.  Using Theorems \ref{thm:good-tree} and \ref{thm:bad-tree}, it will then suffice to show we have packing on all the tree leaves.

Let us write this rigorously.  Recall we have fixed $\rho(n) \leq 1/10$ so that
\[
2c_1(n) c_2(n) \rho \leq 1/2,
\]
where $c_1$ is as in Theorem \ref{thm:good-tree}, $c_2$ is as in Theorem \ref{thm:bad-tree}.

We inductively define for each $i = 0, 1, 2, \ldots$ a family of tree leaves $\{B_{r_f}(f)\}_{f \in \cF_i}$, and stop balls $\{B_{r_s}(s)\}_{s \in \cS_i}$.  Here $r_f$, $r_s$ are radius functions, which may (and do) vary with $f$, $s$; we caution the reader that we only have the \emph{upper bound} $r_f \leq r_i$ at each $i$.  For each $i$, the leaves $\cF_i$ will be either all good balls, or all bad balls.

We let $\cF_0 = \{0\}$, and define the associated radius function $r_{f = 0} = 1$, so the ball $B_1(0)$ is our only leaf at stage $0$.  We let $\cS_0 = \emptyset$, so there are no stop balls at stage $0$.  Trivially, the leaves $\cF_0$ are either all good or all bad.

Suppose we have defined the leaves and stop balls up to stage $i-1$.  The leaves in $\cF_{i-1}$ are (by inductive hypothesis) either all good or all bad balls.  If they good, let us define for each $f \in \cF_{i-1}$ a good tree $\cT_{G, f} = \cT_G(B_{r_f}(f))$, with parameters $\rho$ and $\eta$ as fixed above.  Then we set
\[
\cF_i = \bigcup_{f \in \cF_{i-1}} \cF(\cT_{G, f}),
\]
and
\[
\cS_i = \cS_{i-1} \cup \bigcup_{f \in \cS_{i-1}} \cF(\cT_{G, f}).
\]
Since leaves of good trees are always bad balls, all the leaves $\cF_i$ are bad.

On the other hand, if all the leaves $\cF_{i-1}$ are bad, for each $f \in \cF_{i-1}$ define the bad tree $\cT_{B, f} = \cT(B_{r_f}(f))$, and correspondingly set
\[
\cF_i = \bigcup_{f \in \cF_{i-1}} \cF(\cT_{B, f}), \quad \text{and} \quad \cS_i = \cS_{i-1} \cup \bigcup_{f \in \cF_{i-1}} \cS(\cT_{B, f}).
\]
Clearly all the leaves $\cF_i$ now are good.

This completes the construction.  By concatenating the trees, we obtain the following estimates.
\begin{theorem}\label{thm:all-trees}
There is some integer $N$ so that $\cF_N = \emptyset$, and we have:
\begin{enumerate}[label=(\Alph*)]
\item\label{itm:allleafpacking} packing of \emph{all} leaves:
\[
\sum_{i=0}^{N - 1} \sum_{f \in \cF_i} r_f^k \leq c(n).
\]

\item\label{itm:allstoppacking} Packing of stop balls:
\[
\sum_{s \in \cS_N} r_s^k \leq c(\Lambda, \eps, n, \alpha).
\]

\item\label{itm:allcoveringcontrol} Covering control:
\[
S^k_{\eps, \eta R} \cap B_1(0) \subset \bigcup_{s \in \cS_N} B_{r_s}(s).
\]

\item\label{itm:stopballproperties} Stop ball structure: for any $s \in \cS_N$, we have
\[
\eta R \leq r_s \leq R, \quad \text{and/or} \quad \sup_{B_{2r_s}(s)} W_{2r_s} \leq E - \eta/2.
\]
\end{enumerate}

\end{theorem}

This Theorem directly gives the cover the Theorem \ref{thm:key-packing}.
\begin{proof}[Proof of Theorem \ref{thm:key-packing} given Theorem \ref{thm:all-trees}]
Let $\cU = \cS_N$, and given $x = s \in \cU$, define the radius function
\[
r_x = \max\{ R, r_s \}.
\]
That $\{B_{r_x}(x)\}_{x\in \cS_N}$ is the cover promised by Theorem \ref{thm:key-packing} follows immediately from Theorem \ref{thm:all-trees}\ref{itm:allstoppacking}, \ref{itm:allcoveringcontrol} and \ref{itm:stopballproperties}.
\end{proof}

\begin{proof}[Proof of Theorem \ref{thm:all-trees}]
Let us first show $\cF_N = \emptyset$ for some $N$.  From our tree constructions, in any given (good or bad) tree $\cT(B_r(x))$, \emph{every} leaf $f \in \cF(\cT)$ necessarily satisfies $r_f \leq \rho r$.  We deduce
\[
\max_{f \in \cF_i} r_f \leq \rho \max_{f \in \cF_{i-1}} r_f \leq \rho^i.
\]
For $i$ sufficiently large, we would have $\rho^i < R$, contradicting our definition of a good or bad ball.

Let us prove part \ref{itm:allleafpacking}:  suppose the $\cF_i$ are good.  The collection of $\{f \in \cF_i\}$ are precisely the leaves of bad trees rooted at $\{f' \in \cF_{i-1}\}$.  Therefore, using Theorem \ref{thm:bad-tree}, we have
\[
\sum_{f \in \cF_i} r_f^k \leq 2c_2 \rho \sum_{f \in \cF_{i-1}} r_f^k.
\]
Conversely, if the $\cF_i$ are bad, then the $\cF_i$ are all leaves of good trees rooted at $\cF_{i-1}$.  So by Theorem \ref{thm:good-tree}, 
\[
\sum_{f \in \cF_i} r_f^k \leq c_1 \sum_{f \in \cF_{i-1}} r_f^k.
\]

If $\cF_i$ are good, then it is clear that the $\cF_{i-1}$ are bad, and vice versa.  We deduce that
\[
\sum_{f \in \cF_i} r_f^k \leq c(n) (2c_1 c_2\rho)^{i/2} \leq c(n) 2^{-i/2}.
\]
The packing estimate \ref{itm:allleafpacking} follows directly.

We prove \ref{itm:allstoppacking}.  Each stop ball $s \in \cS_N$ arises from a good or bad tree rooted in some $f \in \cF_i$, for some $i < N$.  Using Theorems \ref{thm:good-tree} and \ref{thm:bad-tree}, we therefore have
\[
\sum_{s \in \cS_N} r_s^k \leq c(\Lambda, \eps, n) \sum_{i=0}^N \sum_{f \in \cF_i} r_f^k \leq c(\Lambda, \eps, n).
\]

We show \ref{itm:allcoveringcontrol}.  Apply Theorems \ref{thm:good-tree} and \ref{thm:bad-tree} to each tree constructed at $f \in \cF_{i-1}$, to deduce
\begin{align*}
\bigcup_{f \in \cF_{i-1}} (S^k_{\eps, \eta R} \cap B_{r_f}(f)) \subset \bigcup_{f \in \cF_i} (S^k_{\eps, \eta R} \cap B_{r_f}(f)) \cup \bigcup_{s \in \cS_i} B_{r_s}(s).
\end{align*}
Since, vacuously, $S^k_{\eps, \eta R} \cap B_1(0) \subset B_1(0)$, it follows by induction that
\[
S^k_{\eps, \eta R} \cap B_1(0) \subset \bigcup_{f \in \cF_i} B_{r_f}(f) \cup \bigcup_{s \in \cS_i} B_{r_s}(s)
\]
for any $i$.  Setting $i = N$ we obtain \ref{itm:allcoveringcontrol}.  \ref{itm:stopballproperties} is immediate from Theorems \ref{thm:good-tree}, \ref{thm:bad-tree}.
\end{proof}


\section{Rectifiability and other Corollaries}

Theorem \ref{thm:main-mass} shows that the Hausdorff measures $\haus^k \llcorner S^k_\eps$ are upper-Ahlfors-regular away from the boundary $\partial \Omega$.  Demonstrating rectifiability is then an easy calculation in the spirit of Lemma \ref{lem:centered-packing}.

\begin{proof}[Proof of Theorem \ref{thm:rectifiable}]
Let $U_\nu = S^k_\eps \cap \{ W_0 \geq \nu \}$.  We will show each $U_\nu$ is rectifiable.  Since rectifiability is a local property, we can assume we are working in a small ball $B_{2r}(x)$, with $x \in U_\nu \cap B_1(0)$, satsifying
\[
|Q + 1/Q|_{C^0(B_{10r}(x))} \leq \Lambda, \quad [Q]_{\alpha, B_{10r}(x)} \leq \eta, \quad \text{and} \quad \sup_{B_{4r}(x)} W_{2r} - \nu \leq \eta,
\]
for some $\Lambda < \infty$, and $\eta$ to be chosen.

The second condition is simply how $Q$ scales with $u$.  The third follows by upper-semi-continuity: otherwise, we would have a sequence $y_i \to x \in U_\nu$, and $r_i \to 0$, with $W_{r_i}(y_i) \geq \nu + \eta$.  But this would imply
\[
W_0(y) \geq \limsup_i W_{r_i}(y_i) \geq \nu + \eta \geq W_0(y) + \eta,
\]
a contradiction.

Take $\delta(n,\Lambda, \eps,\alpha)$ as in Theorem \ref{thm:L2-est}, and $\gamma(n,\Lambda, \delta,\alpha)$ as in Lemma \ref{lem:sym-from-drop}.  We now ensure $\eta \leq (2c_0 + 1)^{-1} \min\{\delta, \gamma\}$.  Write $\mu_\nu = \haus^k \llcorner U_\nu$, and for convenience set $r_i = 2^{-i}$.  Recall that Theorem \ref{thm:main-mass} implies
\[
\mu_\nu(B_s(y)) \leq c(n, \Lambda, \eps, \alpha) s^k \quad \forall s \leq 2r, y \in B_{2r}(x).
\]

We calculate using Theorem \ref{thm:L2-est} and Fubini:
\begin{align*}
&\sum_{r_j \leq r} \int_{B_r(x)} \beta_{\mu_\nu, 2}^k(z, r_j)^2 d\mu_\nu(z) \\
&\leq c \sum_{r_j \leq r} r_j^{-k} \int_{B_r(x)} \int_{B_{r_j}(z)} W_{8r_j}(y) - W_{r_j}(y) + c \eta r^\alpha_j d\mu_\nu(y) d\mu_\nu(z) \\
&\leq c \int_{B_{2r}(x)} \left( \sum_{r_j \leq r} W_{8r_j}(y) - W_{r_j}(y) + c \eta r^\alpha_j \right) d\mu_\nu(y) \\
&\leq c \eta r^k.
\end{align*}
Now we can ensure $\eta$ is small and use Rectifiable-Reifenberg Theorem \ref{thm:rect-reifenberg}, or Theorem 1.1 of \cite{azzam-tolsa}, to deduce $U_\nu$ is rectifiable.

Since rectifiability is stable under countable unions, we obtain rectifiability of each $S^k_\eps$ and $S^k$.
\end{proof}

We now prove the remaining Theorems stated in the Introduction.
\begin{proof}[Proof of Corollary \ref{cor:main-sing}]
Immediate from Proposition \ref{prop:singsetinsidestrata} and Theorems \ref{thm:main-mass}, \ref{thm:rectifiable}.
\end{proof}

\begin{proof}[Proof of Corollary \ref{cor:weak-Lp}]
First, we observe the following:  if $0 \in \partial \fb$, and $\partial \fb$ satisfies the $\eps$-regularity condition \eqref{eqn:eps-reg-cond} in $B_2(0)$, then by Theorem \ref{thm:altcaf} (and in particular the effective estimate of Theorem 8.1 in \cite{alt-caf}), and the higher-regularity of \cite{kinder-niren}, then we have $|D^2 u(0)| \leq c(n)$.

Therefore, using Proposition \ref{prop:fb-almost-sym} and scale-invariance, we deduce there exists some $\eps(n, \Lambda)$ so that $|D^2u(x)| \leq c(n)/r$ whenever $u$ is $(n-k^*+1, \eps)$-symmetric in $B_r(x)$, where $x \in B_1(0)$ and $r \in (0, 1]$.

This implies that
\[
\{ x \in \partial \fb \cap B_1(0) : |D^2u| > 1/r \} \subset S^{n-k^*}_{\eps, c(n)r}(u) \cap B_1(0),
\]
and the required packing estimates follow from Theorem \ref{thm:main-mass}, and the Ahlfors-regularity of the free-boundary (Theorem 4.5 in \cite{alt-caf}).
\end{proof}

Finally, we prove the very first Theorem stated.
\begin{proof}[Proof of Theorem \ref{thm:teaser}]
Rectifiability is direct from Theorem \ref{thm:rectifiable}.  The packing bound is a simple covering argument.  Take $\rho = d(D, \partial D')/20$.  Choose a Vitali cover of $D$ by balls $\{B_{5\rho}(x_i)\}_i$ centered in $D$, so that the $\rho$-balls are disjoint.  In each ball we can apply Theorem \ref{thm:main-mass}.  Since the number of balls is $\leq c(n)\leb^n(D') \rho^{-n}$, we obtain the required packing estimate.
\end{proof}

\bibliographystyle{alpha}
\bibliography{alt-caf}

\end{document}